\documentclass[12pt]{article}
\usepackage{amsmath,amssymb,amsthm,amsfonts,graphicx,color,enumerate,float,epsfig,comment}
  \usepackage{tikz}
  \usepackage{tkz-euclide}
\usepackage[title]{appendix}

  \usetikzlibrary{arrows,calc,decorations.pathmorphing,backgrounds,positioning,fit}

\usepackage{filecontents}
\usepackage{hyperref}
\hypersetup{
    colorlinks=true,
    linkcolor=black,
    citecolor=black,
    }
\DeclareMathOperator{\supp}{{\rm supp}}
\DeclareMathOperator{\id}{1}
\DeclareMathOperator{\Prob}{{\bf P}}
\DeclareMathOperator{\E}{{\bf E}}
\DeclareMathOperator{\Isom}{{\rm Isom}}
\DeclareMathOperator{\Ax}{{\rm Ax}}

\newcommand{\cl}[1]{\lceil #1 \rceil}
\def\N{{\mathbb N}}
\def\e{{\varepsilon}}
\def\H{{\Bbb H}}
\def\Dc{{\mathcal D}}
\def\Cay{{\rm Cay}}

\newcommand{\tr}[1]{{\rm tr}(#1)}
\newcommand{\qa}[1]{{\rm Ax}(#1)}
\newcommand{\nbd}[2]{{\mathcal{N}}_{#2}(#1)}
\newcommand{\mb}[1]{\boldsymbol{#1}}

\begin{document}

\newcommand{\kf}[1]{\marginpar{\tiny #1 --kf}}
\newcommand{\rt}[1]{\marginpar{\bf\tiny #1 --rt}}

\newtheorem{thm}{Theorem}[section]
\newtheorem{lemma}[thm]{Lemma}
\newtheorem{cor}[thm]{Corollary}
\newtheorem{prop}[thm]{Proposition}
\newtheorem{claim}[thm]{Claim}
\newtheorem{question}[thm]{Question}
\newtheorem*{mainthm0}{Theorem \ref{log}}

\newtheorem{remark}[thm]{Remark}
\newtheorem{important remark}[thm]{Important Remark}
\newtheorem{definition}[thm]{Definition}
\newtheorem{example}[thm]{Example}
\newtheorem{fact}[thm]{Fact}
\newtheorem{convention}[thm]{Convention}

\def\ad{{\rm ad}}
\def\diam{\operatorname{diam}}
\def\vl{\operatorname{vl}}
\def\svl{\operatorname{svl}}
\def\cal{\mathcal}
\def\R{{\mathbb R}}
\def\Z{{\mathbb Z}}

\def\Mod{{\rm Mod}}
\def\ep{{\varepsilon}}
\def\C{{\mathcal C}}
\def\T{{\mathcal T}}
\def\TT{\overline{\mathcal T}}
\def\dTT{\partial \overline{\mathcal T}}
\def\S{{\mathcal S}}
\def\SS{{\overline{\mathcal S}}}
\def\M{{\mathcal M}}
\def\L{{\mathcal L}}
\def\D{\Delta}
\def\s{\sigma}
\def\t{\tau}
\newcommand{\cN}{{\cal N}}
\newcommand{\F}{{\mathbb F}}

\def\<{{\langle}}
\def\>{{\rangle}}
\def\red{\textcolor{red}}

\title{Drift and joint minimal displacement for groups acting on hyperbolic spaces}

\author{Koji Fujiwara and Ryokichi Tanaka
\thanks{
K.F.\ is supported by JSPS Grant-in-Aid for Scientific Research 25H00588.}
\thanks{ 
R.T.\ is supported by JSPS Grant-in-Aid for Scientific Research 24K06711.}}
\renewcommand\footnotemark{}

\newcommand{\KF}{
\begin{flushleft}
\textsc{Koji Fujiwara}\\
OIST, Okinawa, 904-0495 JAPAN, and\\ KUIAS/RIMS
Kyoto University, 
JAPAN\\
\textit{E-mail address}: 
{\tt kfujiwara@math.kyoto-u.ac.jp}
\end{flushleft}
}

\newcommand{\RT}{
\begin{flushleft}
\textsc{Ryokichi Tanaka}\\
Department of Mathematics,\\
Kyoto University,
Kyoto 606-8502 JAPAN\\
\textit{E-mail address}: {\tt rtanaka@math.kyoto-u.ac.jp}
\end{flushleft}
}

\maketitle

\begin{abstract}
We obtain a uniform lower bound for the drift of random walks on hyperbolic spaces in terms of joint minimal displacement. Our main theorem applies to weakly acylindrical actions on Gromov-hyperbolic spaces and is quantitative in all relevant constants. For discrete isometry groups of Hadamard manifolds with pinched negative curvature, the required assumptions are automatically satisfied by the Margulis lemma. 

This yields uniform gaps at zero in drift spectra for pinched negatively curved Hadamard manifolds, hyperbolic groups, and mapping class groups acting on curve graphs. 
\end{abstract}

\section{Introduction}

Let $(X, d)$ be a Gromov-hyperbolic metric space.
Let $\mu$ be a countably supported probability measure on $\Isom(X)$.
Consider an independent, identically distributed (i.i.d.)\ sequence of random variables $(g_n)_{n \in \N}$ with the law $\mu$.
Let
$w_n:=g_1\cdots g_n$.
We call $(w_n)_{n \in \N_0}$ a $\mu$-{\bf random walk} starting from the identity $1$ where $w_0:=1$.

Fix a point $o \in X$. 
We call $(w_no)_{n \in \N_0}$  a random walk on $X$.
Suppose that $\mu$ has a finite first moment, i.e., $\E d(o, w_1 o)<+\infty$.
Here $\E$ denotes the expectation with respect to the distribution of $\mu$-random walk $(w_n)_{n\in \N_0}$.
The following limit exists and is called the {\bf drift}:
\[
\ell(\mu):=\lim_{n \to +\infty}\frac{1}{n}\E d(o, w_n o).
\]

If the semigroup generated by the support of $\mu$ contains a pair of independent hyperbolic isometries on $X$, the measure $\mu$ is called {\bf non-elementary}.
The drift is known to be positive if $\mu$ is non-elementary \cite[Theorem 1.2]{MT} and \cite[Theorem 1.1]{G}. 

In this paper we give a uniform lower bound of the drift in terms of 
joint minimal displacement. 
For a finite set $U$ of isometries of $X$, 
the {\bf joint minimal displacement} is defined by
\[
L(U):=\inf_{x \in X}\max_{u \in U}d(x, ux).
\]
See \cite{BF} for a systematic study around  this invariant. 

\subsection{Main results}
We  discuss several settings for $X$.  One is that $X$ is an Hadamard manifold with pinched sectional curvature.
An {\bf Hadamard manifold} is a connected, simply connected, complete Riemannian manifold with non-positive sectional curvature.
We say it has pinched negative sectional curvature if the sectional curvature 
is between two negative constants.

For random walks driven by a uniform distribution $\mu_U$ on a finite set $U$ in $\Isom(X)$, we have the following.

\begin{thm}[Corollary \ref{cor.srwHn}]\label{thm.intro_srw}
Fix integers $n, m \ge 2$ and a real $a\ge 1$.
Then, there exists a constant $\kappa(n, a, m)>0$ such that the following holds.

Let $X$ be an $n$-dimensional Hadamard manifold with sectional curvature $K_X \in [-a^2, -1]$,
and $U$ a finite symmetric set in $\Isom(X)$ of cardinality $m$.

If the group $\<U\>$ generated by $U$ is discrete and non-virtually-nilpotent, then
\[
\ell(\mu_U) \ge \kappa(n, a, m)L(U).
\]
\end{thm}

Here a finitely generated group $G$ is virtually nilpotent if $G$ has a finite index nilpotent subgroup.
See  Corollary \ref{cor.srwHn} for an explicit form of the constant $\kappa(n, a, m)$.

As a consequence of the theorem, we obtain a uniform gap at zero in the set of all values $\ell(\mu_U)$ when we vary $U$ of a fixed size such that $\<U\>$ is discrete. 
Let $\varepsilon(n,a)>0$ be a Margulis constant of an $n$-dimensional Hadamard manifold with sectional curvature $K_X \in [-a^2, -1]$. Then $L(U) \ge \varepsilon(n, a)$
if $\<U\>$ is discrete and non-virtually-nilpotent.

\begin{cor}[Corollary \ref{cor.dichotomy}]\label{cor.intro_dichotomy}
Fix integers $n, m \ge 2$ and a real $a\ge 1$.
Then, there exists a constant $c(n, a, m)>0$ such that the following holds.

Let $X$ be an $n$-dimensional Hadamard manifold with sectional curvature $K_X \in [-a^2, -1]$, and $U$ a finite symmetric set in $\Isom(X)$ of cardinality $m$.

If $\<U\>$ is discrete, then 
either $\ell(\mu_U)=0$ or $\ell(\mu_U) \ge c(n, a, m)$.
\end{cor}

Theorem \ref{thm.intro_srw} is a consequence of our general theorem (Theorem \ref{thm.intro_meta}).
 We now give another consequence in a different setting. Instead of varying the finite set $U$, we fix a finitely generated group $\Gamma$ and a random walk on $\Gamma$, and vary the isometric action of $\Gamma$ on $X$.
 
 A measure $\nu$ on $\Gamma$ is called {\bf admissible} if the support of $\nu$ generates $\Gamma$ as a semigroup.
Let $\rho: \Gamma \to \Isom(X)$ be a discrete, faithful, isometric action
of  $\Gamma$ on $X$.
Let $\ell(\nu, \rho):=\ell(\rho_\ast \nu)$ denote the drift for the pushforward $\mu=\rho_\ast \nu$.

Our next theorem gives a uniform lower bound for this drift, independent of the action $\rho$.

\begin{thm}[Corollary \ref{cor.drift_Hn}]\label{thm.intro_Hn}
Fix an integer $n \ge 2$ and a real $a\ge 1$.
Let $\Gamma$ be a finitely generated group, and 
$\nu$ an admissible probability measure on $\Gamma$ with a finite first moment. 
Let $\Sigma$ be a finite symmetric generating set of $\Gamma$ with $1 \in \Sigma$.

Then, there exists a constant $\kappa(n, a, \Gamma, \Sigma, \nu)>0$ such that the following holds.

Let $X$ be an $n$-dimensional Hadamard manifold with sectional curvature $K_X \in [-a^2, -1]$. 
For every faithful, isometric action $\rho:\Gamma \to \Isom (X)$ such that $\rho(\Gamma)$ is discrete and non-virtually-nilpotent,
we have
\[
\ell(\nu, \rho)\ge \kappa(n, a, \Gamma, \Sigma, \nu)L(\rho(\Sigma)).
\]
\end{thm}

When $n=2$, $a=1$, namely $X=\Bbb H^2$, and $\Gamma$ is the fundamental group
of a closed orientable surface of genus at least two,
this result was proved in \cite[Theorem 3.3]{AGG}.
They also showed that the drift functional $\ell(\nu,\rho)$
is proper on Teichm\"{u}ller space.

Their proof uses a scaling limit argument and gives
no explicit value of the constant $\kappa$.
Our proof is different. It gives an explicit value
of the constant $\kappa$ and extends their result
to a much broader class of actions of finitely generated groups on Hadamard manifolds
with pinched negative curvature. See Corollary \ref{cor.drift_Hn}.

\subsection{The general theorem}

We now state our general theorem for group actions on $\delta$-hyperbolic spaces. It has two main geometric assumptions. One is weak acylindricity of the action, introduced in \cite{D} (see Section \ref{sec.action}). The other is  a uniform lower bound on the joint minimal displacement of a bounded power of the generating set.

Roughly speaking, weak acylindricity gives a uniform upper bound on the amount of overlap between a quasi-axis of a hyperbolic isometry and its translates, in terms of its translation length.

For discrete groups acting on Hadamard manifolds with pinched negative curvature, both assumptions follow from the Margulis lemma and the geometry of negatively curved spaces (see Lemma \ref{lemma.independent} and Corollary \ref{cor.Hn}). This is why they do not appear explicitly in Theorems \ref{thm.intro_srw} and \ref{thm.intro_Hn}.

\begin{thm}[Theorem \ref{thm.meta}]\label{thm.intro_meta}
Fix $\delta, \Dc \ge 0$, integers $N, q\ge 1$, and $\theta_0\in (0, 1]$.
There exists a constant $\kappa(\delta, \Dc, N, q, \theta_0)>0$ such that the following holds.

Let $(X, d)$ be a $\delta$-hyperbolic geodesic space.
Let $U$ be a finite symmetric set in $\Isom (X)$ with $1 \in U$ such that
$L(U^N) \ge K_0 \delta+1$.
Assume that the action of $\<U\>$  on $X$ is non-elementary and $\Dc$-weakly acylindrical.

Let $\mu$ be a countably supported  probability measure on $\Isom (X)$ with a finite first moment 
satisfying $\min_{u\in U}\mu^{\ast q}(u)\ge \theta_0$.

Then, we have
\[
\ell(\mu)\ge \kappa(\delta, \Dc, N, q, \theta_0) L(U).
\]
Here $K_0$ is an absolute constant, and $\mu^{\ast q}$ is the $q$-fold convolution of $\mu$.
\end{thm}

The important point is that the constant $\kappa$ depends only on the parameters in the statement, and not on the individual  action or probability measure. This allows us to obtain uniform estimates when the actions or generating sets vary, as in Theorems  \ref{thm.intro_srw} and \ref{thm.intro_Hn}.


\subsection{Applications}
We now give applications of Theorem \ref{thm.intro_meta} to hyperbolic groups and mapping class groups. In both cases, we obtain uniform lower bounds for the drift as the generating set varies.

We first consider hyperbolic groups. Let $G$ be a hyperbolic group and $U_0$ a finite symmetric generating set of $G$. 
 For a finite symmetric set $U\subset G$, we consider the $\mu_U$-random walk on $G$. 
We use the metric on $\operatorname{Cay}(G,U_0)$ to define both the drift $\ell(\mu_U)$ of random walks on \(G\) and the joint minimal displacement $L(U)$.

\begin{thm}[Corollary \ref{cor.drift.hyp.gp}]
Fix  $\delta \ge 0$ and positive integers $m_0$ and $m$.
There exists a constant $\kappa_1(\delta, m_0, m)>0$ such that the 
following holds. 

Let $G$ be a hyperbolic group, and $U_0$ a finite symmetric generating
set of $G$ of cardinality $m_0$ such that $\Cay(G,U_0)$ is $\delta$-hyperbolic.

For every finite symmetric set $U$ in $G$ of cardinality $m$ such that $\<U\>$ is a non-elementary subgroup of $G$,  we have
\[
\ell(\mu_U) \ge \kappa_1(\delta, m_0, m) L(U).
\]
\end{thm}

We next consider mapping class groups acting on curve graphs, which are hyperbolic but not locally finite. Both $\ell(\mu_U)$ and $L(U)$ are defined using the metric on the curve graph.

\begin{thm}[Theorem \ref{thm.mcg}] \label{thm.intro_mcg}
Let $\mathcal S$ be a closed orientable surface of genus at least two, and let $m$ be a positive integer. 
We consider the action of the mapping class group $\Mod(\mathcal S)$ on the curve graph $\mathcal C(\mathcal S)$.
Then, there exists a constant $\kappa(\mathcal S,m)>0$
such that the following holds:

Let $U$ be a finite symmetric set in $\Mod(\mathcal S)$ of cardinality $m$ such that $\<U\>$ contains a pseudo-Anosov element 
and is not virtually cyclic.  
Then, for the uniform distribution $\mu_U$ on $U$, we have
$$\ell(\mu_U) \ge \kappa(\mathcal S,m) L(U).$$
\end{thm}

As a consequence, we obtain a uniform gap at zero for the drift of random walks on mapping class groups.

\begin{cor}[Corollary \ref{cor.gap.mcg}]
Let $\mathcal S$ be a closed orientable surface of genus at least two, and 
let $m$ be a positive integer. 
Then, there exists a constant $c(\mathcal S, m)>0$ such that the following holds:

For every finite symmetric set $U$ in $\Mod(\mathcal S)$ of cardinality $m$, 
we have
\begin{enumerate}[(1)]
\item
$\ell(\mu_U) \ge c(\mathcal S, m)$, if
$\<U\>$ contains at least one pseudo-Anosov element, and $\< U\>$ is not virtually cyclic.

\item
Otherwise, $\ell(\mu_U)=0$.
\end{enumerate}

\end{cor}

\subsection{Key ideas in the proofs}

Our proof has two main ingredients: a quantitative
construction of Schottky sets and a probabilistic
argument based on the pivotal time technique.

Schottky sets were introduced by Boulanger, Mathieu,
Sert, and Sisto \cite{BMSS}, who used them to establish
large deviation results for random walks on hyperbolic
spaces.

Gouëzel \cite{G} developed a systematic approach using pivotal times and Schottky sets to establish linear progress without any moment assumption, and hence positivity of drift when the first moment is finite.

A Schottky set is a finite set $S$ of isometries with
the following property: for any two points $x,y\in X$,
most elements $s\in S$ make the three points $sx,o,y$
roughly aligned along a geodesic.
The idea of the pivotal time technique is to control
backtracking by using such elements.
Roughly speaking, when a random walk backtracks,
there is a good chance that a different choice of
Schottky elements would make it progress instead.
This leads to linear progress with exponentially
high probability.

Schottky sets can be constructed by a standard
ping-pong argument. For our purpose, however, we
need a construction with uniform estimates in terms
of the joint minimal displacement $L(U)$.
We use ideas from \cite{BF} and weak acylindricity
to obtain such a construction (Proposition
\ref{prop.schottky}).
An important point is to choose the basepoint $o$
carefully, rather than arbitrarily.
Our construction gives uniform control of the word
lengths of the elements of $S$ and of the Schottky
constants in terms of $L(U)$.
This is the main geometric ingredient of our proof.

We then apply Gou\"ezel's pivotal time technique
after rescaling the metric by $L(U)$.
Using the estimates from \cite{G}, we keep track
of the dependence of the constants in the
probabilistic argument.
Combining this with our construction of Schottky
sets gives a uniform exponential bound on the
probability that the random walk fails to make
linear progress, without any moment assumption
(Theorem \ref{thm.meta}).
In particular, we obtain a uniform lower bound
for the drift when the measure has a finite
first moment.
We also obtain an exponential estimate for
translation lengths when the measure has bounded
support (Corollary \ref{cor.tr}).

Finally, we apply these results to discrete groups
acting on Hadamard manifolds with pinched negative
curvature.
By Corollary \ref{cor.Hn}, based on \cite{BF2},
such actions are weakly acylindrical, with a
constant depending only on the dimension and
curvature bounds.
The Margulis lemma and the CAT(0) geometry of
$X$ give the other required uniform estimate
for joint minimal displacement (Lemma
\ref{lemma.independent}).
Together, these facts yield Theorems
\ref{thm.intro_srw} and \ref{thm.intro_Hn}.

\bigskip

The paper is organized as follows. Section \ref{sec.pre} recalls basic notions on hyperbolic spaces. Section \ref{sec.displacement} discusses joint minimal displacement, following \cite{BF}. Section \ref{sec.action} introduces weak acylindricity, and Section  \ref{sec.schottky_main} gives our quantitative construction of Schottky sets in Proposition \ref{prop.schottky}.
In Section \ref{sec.rw}, we review Gou\"{e}zel's pivotal time technique  and obtain an exponential bound on the probability that the random walk fails to make linear progress. Section \ref{sec.drift} contains the proofs of our main results. Finally, Section \ref{sec.applications} gives applications to hyperbolic groups and mapping class groups.

\subsection*{Acknowledgments}
The authors are grateful to Emmanuel Breuillard and Inhyeok Choi for their comments. 
The first author would like to thank The Mathematical Institute 
of  the University of Oxford for hospitality.


\section{Preliminaries}\label{sec.pre}
$\N$ ($\N_0$) denotes the set of positive (non-negative) integers.

\subsection{$\delta$-hyperbolic space}
We refer the reader to \cite{BH} and \cite{GH} as general references for this subsection. 
Let $(X, d)$ be a metric space.
For $x, y, z \in X$,
we define the {\bf Gromov product} by
\[
(x, y)_z:=\frac{1}{2}(d(x, z)+d(z, y)-d(x, y)).
\]
For $\delta\ge 0$,
we say that a metric space $(X, d)$ is {\bf $\delta$-hyperbolic} 
if for all $x, y, z, w \in X$, we have
\[
(x, z)_w \ge \min\{(x, y)_w, (y, z)_w\}-\delta.
\]
We call $X$ {\bf Gromov-hyperbolic} if $X$ is 
$\delta$-hyperbolic for some $\delta$.

A metric space $(X, d)$ is called {\bf geodesic} if every pair of points in $X$ is connected by a geodesic in $X$.
A geodesic triangle is called {\bf $\Delta$-thin} 
if each side is contained in the $\Delta$-neighborhood of the union of the other two sides. If $X$ is $\delta$-hyperbolic, then every geodesic triangle in 
$X$ is $4\delta$-thin (see Lemma \ref{lemma.tripod}).
Indeed, this notion gives another equivalent definition of Gromov-hyperbolicity 
that a geodesic space $X$ is Gromov-hyperbolic if and only if there exists a constant $\Delta$ such that 
every geodesic triangle is $\Delta$-thin. 

Although some of the arguments do not require spaces to be geodesic, 
for simplicity, we assume throughout that spaces are geodesic unless stated otherwise.

We recall  the Gromov boundary of a $\delta$-hyperbolic metric space $X$.
Fix a point $o \in X$.
A sequence $(x_i)_{i\in \N_0}$ in $X$ {\bf converges at infinity} if $(x_i, x_j)_o\to +\infty$ as $i, j\to +\infty$.
Two sequences $(x_i)_{i\in \N_0}$ and $(y_i)_{i\in \N_0}$ converging at infinity are {\bf equivalent} if $(x_i, y_i)_o \to +\infty$ as $i \to +\infty$.
This defines an equivalence relation in the set of sequences converging at infinity by the $\delta$-hyperbolicity of $X$.
The set of equivalence classes is the {\bf Gromov boundary} $\partial X$ of $X$.
The definition of $\partial X$ does not depend on the choice of $o$, following from the $\delta$-hyperbolicity of $X$.
Note that by construction an action by isometries on $X$ naturally extends to 
an action by homeomorphisms on $X\cup \partial X$ when $X\cup \partial X$ is equipped with a natural topology.

\subsection{Chains}
We recall an idea from \cite{G}, which will play an important technical role
in this paper. 
For $C, D\ge 0$,
a sequence of points $x_0, \dots, x_n$ in $X$ is called a $(C, D)$-{\bf chain} 
if one has $(x_{i-1}, x_{i+1})_{x_i}\le C$
for all $0<i<n$, 
and $d(x_i, x_{i+1})\ge D$ for all $0 < i \le n$. 
We quote a basic lemma from \cite[Section 3B]{G}.

\begin{lemma}\label{lemma.chain}
Let $(X, d)$ be a $\delta$-hyperbolic space.
Let $C, R>0$.
\begin{enumerate}[(i)]
\item\label{lemma.chain1}
For all $w, x, y, z \in X$,
if $(w, y)_x\le C$, $(x, z)_y\le C+\delta$, 
and
\[
d(x, y)\ge 2C+2\delta+R,
\]
then 
$(w, z)_x\le C+\delta$.
\item\label{lemma.chain2}
For every $(C, D)$-chain $x_0, \dots, x_n$ with $D \ge 2C+2\delta +R$,
we have
\[
d(x_0, x_n) \ge \sum_{i=0}^{n-1}(d(x_i, x_{i+1})-2(C+\delta))\ge R n.
\]
\item\label{lemma.chain3}
For every $(C, D)$-chain $x_0, \dots, x_n$ with $D \ge 2C+4\delta+R$,
we have
$(x_0, x_n)_{x_i} \le C+2\delta$
for all $i=0, \dots, n$.
\end{enumerate}
\end{lemma}
\proof
The parts \eqref{lemma.chain1}, \eqref{lemma.chain2}, and \eqref{lemma.chain3} are Lemmas 3.5, 3.7, and 3.8 in \cite{G}, respectively, when $R=1$.
The case when $R>0$ is reduced to the case when $R=1$ by considering the  space $(X, (1/R)d)$, which is $(\delta/R)$-hyperbolic.
\qed

\subsection{Quasi-axes and translation lengths of isometries}\label{sec.quasiaxis}

Let $(X,d)$ be a $\delta$-hyperbolic space.
For all $x, y \in X$,
let $[x, y]$ denote a geodesic path between $x$ and $y$.
Let $\Isom (X)$ denote the group of isometries of $X$.
Fix $x \in X$.
For each $g \in \Isom (X)$, we define the {\bf translation length} by
\[
\tr{g}:=\lim_{n \to +\infty}\frac{1}{n}d(x, g^n x).
\]
Here the limit always exists and does not depend on the choice of $x$.
For every positive integer $n$,
we have $\tr{g^n}=n\tr{g}$.

An isometry $g$  is called {\bf hyperbolic} if $\tr{g}>0$.
In the classical case 
$X=\Bbb H^n$, every hyperbolic isometry $g$ has a unique invariant
geodesic, called  the axis of $g$. 
For a general hyperbolic space, an invariant geodesic does not 
always exist. As a substitute, we define a bi-infinite path, called a {\bf quasi-axis}.

\begin{lemma}[Quasi-axis. Lemma 1, \cite{F}]\label{lemma.quasiaxis}
For a hyperbolic isometry $g$ of $X$,
there exists a path (continuous map) $\alpha: \R \to X$ such that
\begin{enumerate}[(i)]
\item $g(\alpha)=\alpha$ and there exists  $T>0$ such that $g(\alpha(t))=\alpha(t+T)$ for every $t \in \R$.
\item Let $t_1, t_2 \in \R$ with $t_1<t_2$.
Then every geodesic between $\alpha(t_1)$ and $\alpha(t_2)$ and the path $\alpha|_{[t_1, t_2]}:[t_1, t_2]\to X$ of $\alpha$ are in the $\Delta$-neighborhoods of each other,
where $\Delta=1672\delta$.
\end{enumerate}
Moreover, $\alpha$ can be parametrized by arc length such that $T>0$.
\end{lemma}
In \cite{F}, the constant $\Delta$ is  $418\delta$,
but the  $\delta$ in that paper refers to the constant of the thin-triangle condition.
Hence we take $\Delta=1672\delta$ in our setting since $X$ only satisfies 
$4\delta$-thin triangle condition (see Lemma \ref{lemma.tripod}).

We call the image of the path $\alpha$ in Lemma \ref{lemma.quasiaxis}
 a {\bf $\Delta$-quasi-axis} of $g$, denoted by $\qa{g}$.
The path $\alpha$ is parametrized by arc length.

For every isometry $g$ of $X$ and every positive integer $n$,
we have
\[
n\tr{g} \le d(x, g^n x)
\]
for all $x \in X$ by definition of the translation length.
For a hyperbolic isometry $g$, the reverse inequality holds up to an additive constant $O(\delta)$ if $x$ belongs to a quasi-axis of $g$.

\begin{lemma}[Lemma 2, \cite{F}]\label{lemma.tr}
Let $g$ be a hyperbolic isometry on a $\delta$-hyperbolic geodesic metric space with a $\Delta$-quasi-axis $\qa{g}$.
Then for every integer $n>0$ and every point $x \in \qa{g}$,
we have
\[
d(x, g^n x)-8\delta-2\Delta \le n\tr{g}.
\]
\end{lemma}

As mentioned above, the constant $2\delta$ in \cite{F} corresponds  to $8\delta$
in the current setting.

For a subset $A$ in $X$ and $x \in X$,
let
$d(x, A):=\inf\{d(x, a) \ : \ a \in A\}$.
The following is a basic lemma which holds for a general $\delta$-hyperbolic geodesic metric space $X$.
In the statement, we assume that $\alpha$ is a quasi-axis; however, the claim holds for all $\alpha$ with the property (ii) in Lemma \ref{lemma.quasiaxis}.

\begin{lemma}\label{lemma.rectangle}
Let $\alpha$ be a $\Delta$-quasi-axis in $X$.
For all $x, y \in X$, let $x', y'$ be points in $\alpha$ such that
\[
d(x, x')= d(x, \alpha) \quad \text{and} \quad d(y, y')= d(y, \alpha).
\]
If $d(x', y')> 10\delta+2\Delta$,
then
\[
d(x, x')+d(x', y')+d(y', y)\le d(x, y)+18\delta+4\Delta.
\]
\end{lemma}

Note that if $\Delta=1672\delta$, then $10\delta+2\Delta\le 4000\delta$ and $18\delta+4\Delta\le 8000\delta$.
Furthermore, we note that for every $\Delta$-quasi-axis $\alpha$ and every $x \in X$, the function $\R \to [0, +\infty)$, $t\mapsto d(x, \alpha(t))$, is proper, i.e., for every real $r\ge 0$, the set of $t \in \R$ with $d(x, \alpha(t))\le r$ forms a compact subset of $\R$.
Hence there exists $x' \in \alpha$ such that $d(x, x')=d(x, \alpha)$.

To show Lemma \ref{lemma.rectangle}, we use the following classical lemma.

\begin{lemma}[The tripod lemma]\label{lemma.tripod}
Let $x, y, z$ in $X$.
For all $y' \in [x, y]$ and all $z'\in [x, z]$ such that $d(x, y')=d(x, z')\le (y,z)_x$,
we have $d(y', z')\le 4\delta$.
\end{lemma}

\proof
See \cite[21.- Proposition, Chapitre 2]{GH}.
\qed

\medskip

Note that it follows that every geodesic triangle is $4\delta$-thin.

\proof[Proof of Lemma \ref{lemma.rectangle}]
First, we have $(x, y')_{x'} \le 4 \delta+\Delta$.
Indeed, otherwise, by Lemma \ref{lemma.tripod}, we have $z_1 \in [x, x']$ with $d(z_1, x')>4\delta+\Delta$,
and $z_2\in [x', y']$ with $d(z_1, z_2)\le 4\delta$.
Since $\alpha$ is a $\Delta$-quasi-axis, $d(z_2, \alpha)\le \Delta$.
Concatenating $[x, z_1]$, $[z_1, z_2]$, and a geodesic from $z_2$ to $\alpha$ yields a strictly shorter path than $[x, x']$, which is impossible.

Similarly, $(x', y)_{y'} \le 4\delta+\Delta$.
If $d(x', y')> 10 \delta+2\Delta$,
then $(x, y)_{y'}\le 5\delta+\Delta$ by Lemma \ref{lemma.chain} (i). 
By definition of the Gromov product,
we obtain
\begin{align*}
d(x, x')+d(x', y')+d(y', y)
\le d(x, y)+18\delta+4\Delta.
\end{align*}
This concludes the proof.
\qed

\section{The joint minimal displacement}\label{sec.displacement}
\subsection{Definition of $L(U)$}

Let $(X, d)$ be a $\delta$-hyperbolic geodesic metric space for $\delta\ge0$.
For a finite set $U$ of isometries in $X$,
the {\bf joint minimal displacement} is defined by
\[
L(U)=\inf_{x \in X}\max_{u \in U}d(x, ux).
\]
We note that $L(U)=L(U\cup\{1\})$,
and that $L(U)=L(U\cup U^{-1})$ \cite[Remark 7.3]{BF}. 
Some basic properties of $L(U)$ are summarized below. See \cite{BF} for a general reference.

\begin{lemma}\label{lemma.disp}
Let $U$, $V$ be finite sets of isometries of $X$.
\begin{enumerate}[(1)]
\item\label{item.UV}
If $U\subset V$, then $L(U) \le L(V)$.
\item\label{item.Un}
For every $n \in \N$,
we have
$L(U^n) \le n L(U)$.
\item\label{item.trL}
For every $u \in U$, we have $\tr{u} \le L(U)$.
\end{enumerate}
\end{lemma}

\proof
 \eqref{item.UV} follows from the definition of $L$.
\eqref{item.Un} follows from the triangle inequality, cf.\ \cite[Claim 3, Section 2.1]{BF}.
We have \eqref{item.trL}  since we have $\tr{u}\le d(x, u x)$ for every $x\in X$ and for every $u \in \Isom (X)$.
\qed

\subsection{The case of $\delta$-hyperbolic spaces}

For $U$ and $n\ge 1$, let $U^n$ denote the set of elements that 
are products of at most $n$ elements of $U$. 
We define
\[
\lambda(U):=\max_{u \in U} \tr{u} 
\quad \text{and} \quad
\lambda_2(U):= \max \left\{\lambda(U), \frac{1}{2} \lambda(U^2)\right\}.
\]

We recall a basic inequality, called geometric Bochi-type inequality, proved in \cite{BF}.

\begin{thm}[Theorem 1.4, \cite{BF}]\label{thm.bochi}
There is an absolute constant $K>0$ such that the following holds.
Let $X$ be a $\delta$-hyperbolic geodesic space and $U \subset \Isom (X)$ a finite 
set. 
Then, we have
\[
L(U) - K\delta \le \lambda_2(U).
\]
\end{thm}

The following is a direct consequence of Theorem \ref{thm.bochi}.
Set 
$$K_0:=2K+4000,$$ where $K$ is the absolute constant in Theorem \ref{thm.bochi}.

\begin{lemma}\label{lemma.independent.2}

Let $X$ be a $\delta$-hyperbolic geodesic space. 
Let $U \subset \Isom(X)$ be a finite set 
with 
\[
L(U) \ge K_0\delta+1.
\]
Then there exists a hyperbolic isometry $g \in U^2$ such that 
\[
\tr{g} \ge L(U)+4000\delta+1.
\]
\end{lemma}

\proof
By Theorem \ref{thm.bochi}, we have $L(U) - K\delta \le \lambda_2(U)$, namely,
there is an element $g \in  U^2$ such that 
$L(U) - K\delta \le \tr{g}/2$.
(Consider $g^2$ if necessary for $g\in U$.)

Since $L(U)\ge (2K+4000) \delta+1$, 
this element  $g \in U^2$ satisfies 
\[
\tr{g} \ge 2L(U)-2K\delta \ge L(U)+4000\delta+1.
\]
Since $\tr{g}>0$, the element $g$ is a hyperbolic isometry on $X$. 
\qed

\subsection{The case of Hadamard manifolds with pinched negative curvature}\label{sec.hadamard}

In the case when a $\delta$-hyperbolic space $X$ is ${\rm CAT}(0)$,
we have more detailed information regarding the joint minimal displacement.
A geodesic metric space $(X, d)$ is ${\rm CAT}(0)$ if and only if every geodesic triangle on $x, y, z$ satisfies the following inequality: 
\[
2d(x, m)^2\le d(x, y)^2+d(x, z)^2-\frac{1}{2}d(y, z)^2,
\]
where $m$ is the midpoint of $[y, z]$.
A simply connected Riemannian manifold with non-positive sectional curvature is a ${\rm CAT}(0)$ space.
See \cite{BH} for a general reference. 
We quote a lemma.

\begin{lemma}[Lemma 1.7, \cite{BF}]\label{lemma.CAT0}
If $X$ is a ${\rm CAT}(0)$ metric space and $U$ a finite set of isometries, 
then for every $m\in \N_0$, we have
\[
L(U^m)\ge \frac{\sqrt{m}}{2}L(U).
\]
\end{lemma}

Given $n \in \N$ and $a \in [1, +\infty)$, 
let $X$ be an $n$-dimensional Hadamard manifold with pinched sectional curvature $K_X \in [-a^2, -1]$.
Note that the hyperbolicity constant $\delta$ of $X$ is at most the one for $\H^2$, see \cite[1.2 Proposition, Chapter III.H]{BH}.

We obtain a lemma from the lemma above and  the Margulis lemma.

\begin{lemma}\label{lemma.independent}
Fix $n\in \N$ and $a \in [1, +\infty)$.
Let $X$ be an $n$-dimensional Hadamard manifold with sectional curvature $K_X \in [-a^2, -1]$.
Let $\delta$ be the hyperbolicity constant of $\Bbb H^2$, so that $X$ 
is $\delta$-hyperbolic. 

Then, there exists a constant $N=N(n, a)$ such that the following holds:
For every finite set $U$ in $\Isom(X)$ with $1 \in U$ such that $\<U\>$ is discrete and 
non-virtually-nilpotent, 
we have
\[
L(U^N) \ge K_0 \delta+1.
\]
Here $K_0$ is the absolute constant in Lemma \ref{lemma.independent.2}.
\end{lemma}

\proof
Let $\e(n, a)$ denote the Margulis constant of $X$.
The Margulis lemma states that if $\<U\>$ is discrete and non-virtually-nilpotent,
then $L(U)\ge \e(n, a)$, see \cite[Section 8.3]{BGS}.
By Lemma \ref{lemma.CAT0},
we have
\[
L(U^m)\ge \frac{\sqrt{m}}{2}L(U) \ge \frac{\sqrt{m}}{2}\e(n, a),
\]
for every $m \in \N_0$.

Letting 
\[
N:=\left\lceil\left(\frac{2K_0\delta+2}{\e(n, a)}\right)^2 \right\rceil,
\]
we obtain the claim since $N$ depends only on $n$ and $a$.
Here, we denote by $\cl{r}$ the smallest integer greater than or equal to $r \in \Bbb R$.
\qed

\section{Weak acylindricity of group actions}\label{sec.action}

We recall that two hyperbolic isometries $g$ and $h$ on a $\delta$-hyperbolic space $X$ are called
{\bf independent} if the sets of fixed points of $g$ and $h$ in $\partial X$ are disjoint.
Let $G$ be a (countable) subgroup of $\Isom (X)$.
We call the isometric action of $G$ on $X$ {\bf non-elementary} if $G$ contains two hyperbolic elements $g$ and $h$ that are independent. 
Equivalently, $G$ has no bounded orbit in $X$, no fixed point in $\partial X$, and no invariant two-point subset of $\partial X$; see \cite[Proposition 3.1]{CCMT}
(they call it {\it general type}).

For a hyperbolic isometry $g \in G$, the {\bf elementary closure} of $g$ in $G$
is defined  by
\[
E(g):=\big\{h \in G \ : \ d_{\rm Haus}(h\qa{g}, \qa{g})<+\infty\big\}.
\]
Here $d_{\rm Haus}$ denotes the Hausdorff distance and $\qa{g}$ is a $\Delta$-quasi-axis of $g$.
The definition of $E(g)$ is independent of the choice of $\qa{g}$ and the constant $\Delta$ as two quasi-axes of a hyperbolic isometry are within a bounded Hausdorff distance.
Note that $E(g)$ is a subgroup of $G$, and that $E(g) \neq G$ if $G$ is non-elementary.

\subsection{Weak acylindricity}
We recall that an isometric action of $G$ on $X$ is called {\bf acylindrical} \cite{Btight}
if there exist $N, K$ such that for all pair of points $x,y$ with $d(x,y) \ge K$, 
the set $\{ g\in G \ : \  d(gx,x) + d(gy,y) \le 1000 \delta\}$ has cardinality at most $N$.
In particular, a properly discontinuous and cocompact action is acylindrical since the number of elements $g$ with $d(gx,x) \le 1000 \delta$ is 
uniformly bounded in $x$.

The {\bf weak acylindricity} of a group action  on a hyperbolic space
was defined by Delzant \cite[Definition 6]{D}.

For a ($\Delta$-)quasi-axis $\alpha$ in $X$ and  $x \in X$,
we define the projection $\pi_\alpha(x)$ of $x$ to $\alpha$
as the set of points $a \in \alpha$ such that $d(x, a)=d(x, \alpha)$.
For a subset $A$ in $X$, we denote its diameter by $\diam A$.

\begin{definition}[Weak acylindricity]\label{def.wa}
An isometric action by a group $G$ on a $\delta$-hyperbolic space $X$ 
 is called {\it weakly acylindrical} if
there exists a constant $\Dc > 0$ such that for every hyperbolic element $g \in G$, 
and every $h \not \in E(g)$, we have
\[
\diam \left(\pi_{\qa{g}}(h\qa{g})\right) \le \Dc\left(1+\tr{g}\right).
\]
We call it $\Dc$-weakly acylindrical if we want to make the constant $\Dc$ explicit.  
\end{definition}

The definition involves a choice of $\Delta$-quasi-axis.
Although the constant $\Dc$ may depend on $\Delta$, the property itself is independent of both $\Delta$ and a choice of quasi-axis, up to increasing $\Dc$ by  a constant depending only on $\delta$.
This is a consequence of elementary $\delta$-hyperbolic geometry.

One can say that weak acylindricity is a uniform version of the WWPD property in \cite[Definition 4.18]{BBF}.
Note that an acylindrical action on a hyperbolic geodesic space is weakly acylindrical, such that the weak acylindricity constant 
$\Dc$ depends only on $\delta$ and the constants $N,K$ in the definition of acylindrical action (see \cite[Proposition 5]{D} and its proof).

A non-trivial and important class of examples of weak acylindrical actions is the following. 
The action of a discrete group, $G$, of isometries in a Hadamard manifold with pinched negative curvature is $\Dc$-weakly acylindrical such that the constant $\Dc$ depends only on the dimension and the curvature bound.  See Corollary \ref{cor.Hn}.

\subsection{Projections and neighborhoods of quasi-axes}
In a $\delta$-hyperbolic space, the projection to a quasi-axis
can be described using a neighborhood of the quasi-axis. 
For a subset $A$ in $X$ and $r\ge 0$,
let $\nbd{A}{r}$ denote the closed $r$-neighborhood of $A$ in $X$.
By convention the diameter of the empty set is $0$.
In the following inequalities, diameters are allowed to be infinite.

\begin{lemma}\label{lemma.diam}
Let $\alpha$, $\beta$ be $\Delta$-quasi-axes of some hyperbolic isometries in a $\delta$-hyperbolic geodesic metric space $X$.
For every $R\ge 0$, we have
\[
\diam(\alpha\cap \nbd{\beta}{R}) \le \diam \pi_\alpha(\beta) +4R.
\]
Conversely, we have
\[
\diam \pi_\alpha(\beta) \le \diam(\alpha\cap \nbd{\beta}{20\delta+2\Delta})+10\delta+2\Delta.
\]
\end{lemma}

\proof
We prove the first inequality.
We assume that $\alpha\cap \nbd{\beta}{R}\neq \emptyset$ since otherwise the inequality holds.
For every $x \in \alpha \cap \nbd{\beta}{R}$, 
let $y \in \beta$ with $d(x, y) \le R$.
Since $d(y, \pi_\alpha(y)) \le R$,
we have $x \in \nbd{\pi_\alpha(\beta)}{2R}$.
Hence
\[
\alpha \cap \nbd{\beta}{R} \subset \nbd{\pi_\alpha(\beta)}{2R},
\]
implying the first inequality.

We show the second inequality.
Let $a_1, a_2\in \alpha$. 
Further, let $b_1, b_2\in \beta$ be points such that $a_i \in \pi_{\alpha}(b_i)$, $i=1, 2$.
We consider a geodesic rectangle on $a_1, a_2, b_2, b_1$.

Suppose $d(a_1, a_2)>10\delta+2\Delta$.
As in the proof of Lemma \ref{lemma.rectangle}, we have
\[
(b_1, a_2)_{a_1} \le 4\delta+\Delta \quad \text{and} \quad (b_1, b_2)_{a_2} \le 5\delta+\Delta.
\]
Thus,
by definition of the Gromov product, we obtain
\[
(a_1, b_1)_{a_2}=d(a_1, a_2)-(b_1, a_2)_{a_1} \ge d(a_1, a_2)-4\delta-\Delta>5\delta+\Delta\ge (b_1, b_2)_{a_2}.
\]
Consider geodesic triangles on $a_1, a_2, b_1$, and on $b_1, a_2, b_2$.
By using the tripod lemma (Lemma \ref{lemma.tripod}), 
we have
\[
d(a_1, \beta) \le 12 \delta+2\Delta \quad \text{and} \quad d(a_2, \beta)\le 9\delta+2\Delta.
\]
Thus, $a_1, a_2 \in \alpha \cap \nbd{\beta}{12\delta+2\Delta}$.
Therefore, if $\diam \pi_\alpha(\beta)>10\delta+2\Delta$,
then
$\diam \pi_\alpha(\beta) \le \diam (\alpha \cap \nbd{\beta}{20\delta+2\Delta})$.
The desired second inequality follows by
combining the case when $\diam \pi_\alpha(\beta)\le 10\delta+2\Delta$,
completing the proof.
\qed

\subsection{Discrete group actions on Hadamard manifolds}

We discuss important examples of group actions on Hadamard manifolds that are weakly acylindrical in some uniform sense. 
Given $n \in \N$ and $a \in [1, +\infty)$,
let $X$ be an $n$-dimensional Hadamard manifold with  sectional curvature $K_X \in [-a^2, -1]$.
On an Hadamard manifold, a hyperbolic isometry $g$ has a unique  invariant geodesic, see
\cite{BH}, and 
we take this geodesic as a quasi-axis of $g$. 
Hence $\Delta=0$ in the definition of a $\Delta$-quasi-axis.

We use the following result  from \cite{BF2}, a manuscript in preparation.
Here  $G<\Isom(X)$ being discrete means that the action of $G$ on $X$
is properly discontinuous. 
The proof relies on the Margulis lemma and the geometry of CAT($-1$) spaces.

\begin{prop}[\cite{BF2}]\label{prop.overlap}
There exists a constant $Z>0$ that depends only on $n$ and $a$ with
the following property. 
Let $G < \Isom(X)$ be a discrete group on $X$, 
and $g\in G$ hyperbolic.   
For every $h\in G$ such that $ h \not\in E(g)$, 
we have
\[
\diam \left(\qa{g} \cap \nbd{h\qa{g}}{50\delta}\right) \le (n+2) \tr{g} + Z.
\]
\end{prop}

As a direct application of this proposition, we obtain the following.
This is known under the assumption that $G$ is torsion free
\cite[Lemma 6.2]{WY}.

\begin{cor}[Weak-acylindricity]\label{cor.Hn}
Let $n \in \N$ and $a \in [1, +\infty)$.
Let $X$ be an $n$-dimensional Hadamard manifold $X$ with sectional curvature $K_X\in [-a^2, -1]$.
Let $G <\Isom(X)$ be a discrete subgroup. 

Then there exists a constant $\Dc=\Dc(n, a)$ such that the action of $G$ on $X$ is $\Dc$-weakly acylindrical.
\end{cor}

\proof
Let $\delta$ be a hyperbolicity constant of $\mathbb H^2$.
Then $X$ is $\delta$-hyperbolic.
As we said, since the axis of a hyperbolic isometry of $X$ is
a geodesic, we may take $\Delta=0$ in Lemma \ref{lemma.diam}.

Let $g\in G$ be hyperbolic and $h\in G\setminus E(g)$.
By Lemma \ref{lemma.diam} and Proposition \ref{prop.overlap}, we have
\[
\begin{aligned}
\diam\pi_{\qa g}
  (h\qa g)
&\le
\diam \left(
  \qa g \cap
  \nbd{h\qa g}{20\delta} \right)
  +10\delta\\
&\le (n+2) \tr g+Z+10\delta.
\end{aligned}
\]
Setting
$
\mathcal D=\max\{n+2,Z+10\delta\},
$
we obtain the claim. Note that $\mathcal D$ depends
only on $n$ and $a$.
\qed

\medskip

We record an immediate consequence of Corollary \ref{cor.Hn},
which is also a well-known fact.

\begin{lemma}\label{lemma.nonelem}
Keep the same setting as in Corollary \ref{cor.Hn}.
For every finite set $U$ in $\Isom(X)$,
if $\<U\>$ is discrete and non-virtually-nilpotent, 
then the action of 
$\<U\>$ is non-elementary on $X$.
\end{lemma}

\proof
By Lemmas \ref{lemma.independent} and \ref{lemma.independent.2}, $\<U\>$ contains a hyperbolic element $g$. Let $E(g)$ be the elementary closure in $\<U\>$. 
Then, $E(g)$ is virtually cyclic since it preserves $\Ax(g)$
and $E(g)$ is discrete.
It implies that $\< U\> \not\subset E(g)$
since $\<U\>$ is not virtually-nilpotent. 

Let $u \in \<U \> \backslash E(g)$. 
Set $h=ugu^{-1} \in \< U\>$, which is hyperbolic. Then $g$ and $h$ are independent. 
Indeed, otherwise, $g$ and $h$ fix a common point in $\partial X$.
Then, $\diam \pi_{\Ax(g)} (u \Ax(g))$ is infinite. 
This contradicts Corollary \ref{cor.Hn} since it in particular says that 
$\diam \pi_{\Ax(g)} (u \Ax(g))$ is finite.
\qed

\section{Schottky sets}\label{sec.schottky_main}
In this section, we construct a Schottky set with the additional properties. 
The goal is to establish Proposition \ref{prop.schottky}.

\subsection{The definition of Schottky set}\label{sec.schottky_basic}
We recall the definition of a Schottky set.
It is introduced in \cite{BMSS}, however, we use the version in \cite{G}.
Let $X$ be a $\delta$-hyperbolic space.
A Schottky set is, roughly speaking,  a finite set $S$ of isometries of $X$ with the following property: for all $x$ and $y$ in $X$,
and for all $s$ but a small proportion of the elements in $S$, the triple of points $x$, $o$, $sy$ are roughly aligned along a geodesic.
For a general (possibly non-geodesic) metric space, this condition is expressed in terms of the Gromov product.

For a set $A$, we denote by $\#A$ its cardinality.

\begin{definition}
Fix a point $o \in X$.
Let $\eta, C, D \ge 0$. 
A finite set $S$ of isometries of $X$ is $(\eta, C, D)$-{\bf Schottky} at $o$ if the following holds:
\begin{itemize}
\item For all $x, y \in X$, 
we have
$\#\{s \in S \ : \ (sx, y)_o \le C\} \ge (1-\eta)\#S$.
\item For all $x, y \in X$,
we have 
$\#\{s \in S \ : \ (s^{-1}x, y)_o \le C\} \ge (1-\eta)\#S$.
\item 
For all $s\in S$, 
we have $d(o, so) \ge D$.
\end{itemize}
\end{definition}

The existence of Schottky sets is well known.
Suppose we are given $o \in X$ and a pair of independent hyperbolic isometries $u$ and $v$ of $X$.
Then, for every $\eta>0$, there exists $C>0$ such that for all $D>0$  there exist $n\in \N$ and an $(\eta, C, D)$-Schottky set at $o$ in the set  $\{ w_1\cdots w_n:w_i \in \{u, v\}\}$ \cite[Proposition 3.12]{G}.
This follows from a classical ping-pong method.

For our purpose, however, we need to construct a Schottky set with uniform control on $C$ and $D$ in terms of the joint minimal displacement $L(U)$ of a given finite set $U$ of isometries.
Also, each element in the Schottky set needs to have uniformly bounded length as a word in $U$.
In addition, we need to choose an appropriate point $o$ in the definition of Schottky set to achieve our purpose.
We will state the precise properties of Schottky sets that we need and give the construction in Section \ref{sec.schottky}.

We will use the following two lemmas to construct a Schottky set in Section \ref{sec.schottky}.
The first lemma is a simple application of Lemma \ref{lemma.chain} on chains.

\begin{lemma}\label{lemma.free}
Fix $o \in X$.
Let $C, D > 0$. 
Let $\{u, u'\} \subset \Isom(X)$.
Suppose that the following holds:
\[
\max\big\{(so, s'o)_o \ : \ s, s' \in \{u, u^{-1}, u', u'^{-1}\}, s\neq s'\big\} \le C,
\]
and 
\[
\min\big\{d(o, u o), d(o, u' o)\big\}\ge D.
\]
If $D > 2C+2\delta$, then $\{u, u'\}$ freely generates a free subgroup in $\Isom(X)$.
\end{lemma}

\proof
Let  $w$ and $w'$ be non-empty reduced words on $\{u, u^{-1}, u', u'^{-1}\}$ of length $n$ and $m$, respectively:
\[
w=s_1\cdots s_n \quad \text{and} \quad w'=s'_1\cdots s'_m
\]
with $s_i, s_i'\in \{u, u^{-1}, u', u'^{-1}\}$.
We may assume that $n \le m$.

Let $k$ be the largest $i \le n$ such that $s_j=s_j'$ for all $j \in \{1, \dots, i\}$ if it exists, $k:=0$ if otherwise.
If $w\neq w'$ as words,
then $k<m$,
and the following sequence forms a $(C, D)$-chain:
\[
s_{k+1}\cdots s_n o,
s_{k+1}\cdots s_{n-1} o, \dots, s_{k+1}o, o, s'_{k+1}o, \dots, s'_{k+1}\cdots s'_m o.
\]
Since $k<m$, the sequence consists of at least two points.
If $D > 2C+2\delta$, we have $wo\neq w'o$ by Lemma \ref{lemma.chain} \eqref{lemma.chain2}.
This implies that $w\neq w'$ as isometries.
Therefore, the inclusion $\{u, u'\} \to \Isom (X)$ extends to an injective homomorphism from the free group on $\{u, u'\}$ to $\Isom(X)$, yielding the claim.
\qed

\medskip

The following lemma from \cite{AGG} gives a useful criterion for verifying the Schottky 
set condition.

\begin{lemma}[Lemma 4.1, \cite{AGG}]
\label{lemma.criterion}
For a finite set $S$ of isometries of $X$ with $S=S^{-1}$,
let
\[
C_0=\max_{s_1\neq s_2, s_1, s_2 \in S}(s_1o, s_2o)_o
\quad
\text{and}
\quad 
D_0=\min_{s \in S}d(o, so).
\]
If $D_0>2C_0+4\delta$,
then
$S$ is $(\eta, C_0+3\delta, D_0)$-Schottky at $o$ with
$\eta=2/\# S$. 
\end{lemma}

\subsection{Construction of Schottky sets}\label{sec.schottky}

The following proposition gives a quantitative construction of Schottky sets. 

\begin{prop}\label{prop.schottky}
Let $N \in \N$ and $\delta, \Dc \ge 0$.
Let $X$ be a $\delta$-hyperbolic geodesic space. 
Let $U$ be a finite set of isometries in $X$ with $1 \in U$ and $U=U^{-1}$. 
Let $g$ be a hyperbolic isometry of $X$.
Suppose that the following conditions hold:
\begin{enumerate}[(i)]
\item\label{assump.h}  $U^N$ contains $g$ such that 
$\tr{g} \ge L(U)+4000\delta+1$.
\item\label{assump.action}
$\<U\>$ does not preserve the set of fixed points of $g$ in $\partial X$.

\item\label{assump.wa} The action of $\langle U\rangle$ on $X$ is $\Dc$-weakly acylindrical.
\end{enumerate}
Then there exists a positive integer $M=M(\delta, \Dc, N)$, a point $o\in X$, and positive constants $C(\delta, \Dc, N)$ and $D(\delta, \Dc, N)$ such that the following holds:

\begin{itemize}
\item 
We have $D\ge 20C+100\delta+L(U)$ for  
\[
C=C(\delta, \Dc, N)(L(U)+1) \quad \text{and} \quad D=D(\delta, \Dc, N)(L(U)+1).
\]

\item $U^M$ contains a $(1/100, C, D)$-Schottky set $S$ at $o$. 

\item We have $\# S \ge 200$.
\end{itemize}
\end{prop}

Note that the constants $M(\delta, \Dc, N)$, $C(\delta, \Dc, N)$, and $D(\delta, \Dc, N)$
depend only on $\Dc$, $N$, and $\delta$. 

Under the condition (\ref{assump.wa}),  the condition (\ref{assump.action}) is equivalent
to the action of $\langle U\rangle$ on $X$ being
non-elementary.

The proof of the proposition will be given after several preliminary results. 
Let $g$ be a hyperbolic element satisfying the assumption \eqref{assump.h}.
By the assumption \eqref{assump.action}, 
there exists $u \in U$ such that 
$u \notin E(g)$,
where $E(g)$ denotes the elementary closure of $g$ in $\<U\>$.
Indeed, if otherwise, then the group $\langle U\rangle$ preserves the set of fixed points of $g$, contradicting the assumption \eqref{assump.action}.
Let us define
\[
h:=u g u^{-1} \in U^{N+2}.
\]
Note that we have assumed $U=U^{-1}$, which guarantees that $h \in U^{N+2}$.

By definition of $h$,
we have $\tr{g}=\tr{h}$.
Furthermore, $u\qa{g}$ defines a $\Delta$-quasi-axis of $h$, which we denote by $\qa{h}$.

Since by definition $L(\{g, h\})=\inf_{x\in X}\max\{d(x, g x), d(x, h x)\}$, 
there exists $o=o_{U, N} \in X$ such that
\begin{equation}\label{eq.basepoint}
\max\left\{d(o, g o), d(o, h o)\right\} \le L(\{g, h\})+1.
\end{equation}

The choice of the basepoint $o$ is important in what follows. 
We choose $o$ so that it almost realizes the infimum in the definition of $L(\{g, h\})$. This choice will be used in  the next lemma. 
Strictly speaking, $o$ depends on $g$ and $h$, however, we write $o=o_{U, N}$ for simplicity since $g, h$ are contained in $U^{N+2}$.

\begin{lemma}\label{lemma.tube}
Under the assumptions of Proposition \ref{prop.schottky},
we have
\begin{equation}\label{eq.tube1}
o\in \nbd{\qa{g}}{\Lambda}\cap \nbd{\qa{h}}{\Lambda},
\end{equation}
where $\Lambda=NL(U)+4000\delta+1$. 
Furthermore, for every $m \in \Z$, we have
\[
\max\{(g^m o, h^m o)_o, (g^m o, h^{-m} o)_o\} \le \Dc(\tr{g}+1)+10\Lambda+100\delta.
\]
\end{lemma}

\proof
We first prove that $o \in \nbd{\qa{g}}{\Lambda}$.
Suppose, to the contrary, that $o\notin \nbd{\qa{g}}{\Lambda}$.
Since $1\in U$, we have $\{g, h\} \subset U^{N+2}$, and thus, by Lemma \ref{lemma.disp} \eqref{item.Un},
\[
L(\{g, h\}) \le L(U^{N+2}) \le (N+2)L(U).
\]
On the one hand, 
by \eqref{eq.basepoint},
\begin{align}\label{eq.tube2}
d(o, g o)	
\le L(\{g, h\})+1
\le (N+2)L(U)+1.
\end{align}
On the other hand,
since $o \notin \nbd{\qa{g}}{\Lambda}$,
we have
\begin{equation}\label{eq.tube3}
2\Lambda+\tr{g} \le d(o, g o)+8000\delta.
\end{equation}
To see this, 
choose the points $o, x, g x, g o$ where $x \in \qa{g}$ with $d(o, x)= d(o, \qa{g})$.
Note that $d(x, g x)\ge \tr{g} \ge 4000\delta$ by definition of $\tr{g}$ and the assumption \eqref{assump.h}.
Then, Lemma \ref{lemma.rectangle} shows \eqref{eq.tube3}.

Combining \eqref{eq.tube2} and \eqref{eq.tube3},
we obtain
\[
2\Lambda+\tr{g} \le (N+2)L(U)+8000\delta+1.
\]
Since $\tr{g} \ge L(U)+4000\delta+1$ by the assumption \eqref{assump.h},
this implies 
\[
2\Lambda\le (N+1)L(U)+4000\delta.
\]
This is a contradiction since $\Lambda=NL(U)+4000\delta+1$.
Thus, we have shown $o \in \nbd{\qa{g}}{\Lambda}$.

Similarly, $o \in \nbd{\qa{h}}{\Lambda}$ since $\tr{h}=\tr{g}$ and \eqref{eq.basepoint} holds.
Thus, we obtain \eqref{eq.tube1}.

It suffices to prove the second claim for $g^m$ and $h^m$ with $m \ge 0$
since the case for $m <0$ follows by replacing $g$ and $h$ with their inverses.

We show that
for every $m\ge 0$, 
\begin{equation}\label{eq.tube4}
(g^m o, h^m o)_o -2\Lambda -8\delta\le \diam (\qa{g}\cap\nbd{\qa{h}}{2\Lambda+20\delta}).
\end{equation}
We consider a geodesic triangle on $o$, $g^m o$, and $h^m o$.
First, by \eqref{eq.tube1},
we have $x \in \qa{g}$ and $y \in \qa{h}$ such that $d(o, x) \le \Lambda$ and $d(o, y) \le \Lambda$.
Furthermore, by the property of quasi-axes (Lemma \ref{lemma.quasiaxis} (i)),
we have $g^m x \in \qa{g}$ and $h^m y \in \qa{h}$.
Next, by the tripod lemma (Lemma \ref{lemma.tripod}),
there exist $p \in [o, g^m o]$ and $q \in [o, h^m o]$ such that
\[
d(o, p)=d(o, q)=(g^m o, h^m o)_o \quad \text{and} \quad d(p, q) \le 4\delta.
\]
Furthermore, since 
$d(g^m o, g^m x) \le \Lambda$ and $d(h^m o, h^m y)\le \Lambda$,
we have
\[
d(p, \qa{g}) \le \Lambda+8\delta \quad \text{and} \quad d(q, \qa{h}) \le \Lambda+8\delta.
\]
The first inequality follows by considering a geodesic rectangle on $o, x, g^m x, g^m o$ and applying the tripod lemma (Lemma \ref{lemma.tripod}) twice, and the second inequality follows similarly.
Let $p' \in \qa{g}$ and $q' \in \qa{h}$ with $d(p, p') \le \Lambda+8\delta$ and $d(q, q')\le \Lambda+8\delta$, respectively.
By the triangle inequality, we have
$d(x, y) \le 2\Lambda$ and $d(p', q') \le 2\Lambda+20\delta$.
Thus, we obtain
\[
x, p' \in \qa{g} \cap \nbd{\qa{h}}{2\Lambda+20\delta}.
\]
Since by the triangle inequality $d(o, p) \le d(x, p')+2\Lambda+8\delta$, we have
\[
(g^m o, h^m o)_o-2\Lambda-8\delta \le d(x, p')\le \diam\left(\qa{g}\cap \nbd{\qa{h}}{2\Lambda+20\delta}\right).
\]
This shows \eqref{eq.tube4}.

We now use the $\Dc$-weak acylindricity of the isometric action of $\langle U\rangle$ on $X$ in the assumption \eqref{assump.wa}.
By Definition \ref{def.wa} and Lemma \ref{lemma.diam},
we have:
\begin{align*}
\diam(\qa{g}\cap\nbd{\qa{h}}{2\Lambda+20\delta})
&\le \diam (\pi_{\qa{g}}{\qa{h}})+4(2\Lambda+20\delta)\\
&\le \Dc(\tr{g}+1)+8\Lambda+80\delta.
\end{align*}
Rearranging the terms, we obtain the claim.
\qed

\medskip

We will construct the required Schottky set by using powers of $g$ and $h$ as ``building blocks".
The following lemma gives the key estimates needed for this construction.

\begin{lemma}\label{lemma.separation}
Under the assumptions of Proposition \ref{prop.schottky},
we have the following:
For every $m \in \Z$, 
we have 
\begin{equation}\label{eq.qa}
\max\left\{(g^m o, g^{-m}o)_o, (h^m o, h^{-m}o)_o\right\} \le C_0(\delta, N)(L(U)+1),
\end{equation}
and
\begin{equation}\label{eq.wa}
\max\left\{(g^m o, h^m o)_o, (g^m o, h^{-m} o)_o\right\} \le C_1(\delta, \Dc, N)(L(U)+1),
\end{equation}
where $C_0(\delta, N)=N+10^4\delta+1$ and
$C_1(\delta, \Dc, N)=\Dc N+10N+10^5\delta+10$.

Furthermore, for every $m \in \Z$, we have 
\begin{equation}\label{eq.D}
\min\big\{d(o, g^m o), d(o, h^m o) \big\}\ge |m|(L(U)+1),
\end{equation}
and
\begin{equation}\label{eq.Dup}
\max\big\{d(o, g^m o), d(o, h^m o)\big\} \le |m|((N+2)L(U)+1).
\end{equation}
\end{lemma}

\proof
First, we show \eqref{eq.qa}.
Let us show the required bound on $(g^m o, g^{-m}o)_o$ for every integer $m \ge 1$, 
as the Gromov product equals $0$ when $m=0$.
The bound on $(h^m o, h^{-m}o)_o$ follows analogously.

By Lemma \ref{lemma.tube}, there exists $x \in \qa{g}$ such that $d(o, x)\le \Lambda$.
For every $m\ge 1$, 
since $\tr{g} > 4000\delta$ by the assumption \eqref{assump.h},
we have
\[
d(g^{-m}x, g^m x)=d(x, g^{2m} x) \ge \tr{g^{2m}}=2m\tr{g} > 8000m\delta.
\]
Applying Lemma \ref{lemma.rectangle} to $g^{-m}o$, $g^{-m}x$, $g^m x$, $g^m o$,
we obtain 
\[
2d(o, x)+d(g^{-m}x, g^m x) \le d(g^{-m}o, g^m o)+8000\delta.
\]
Since $g^{-m}x, x, g^m x$ belong to $\qa{g}$, and $\qa{g}$ is a $\Delta$-quasi-axis,
the $\Delta$-neighborhood of a geodesic $[g^{-m}x, g^m x]$ contains $x$.
Thus, we have
\[
d(g^{-m}x, x)+d(x, g^m x) \le d(g^{-m}x, g^m x)+2\Delta.
\]
Therefore, we obtain
\[
2d(o, x)+d(g^{-m}x, x)+d(x, g^m x) \le d(g^{-m}o, g^m o)+8000\delta+2\Delta.
\]
By definition of the Gromov product and the triangle inequality,
we obtain:
\begin{align*}
(g^{-m}o, g^m o)_o
&=\frac{1}{2}(d(o, g^{-m}o)+d(o, g^m o)-d(g^{-m}o, g^m o))\\
&\le \frac{1}{2}(4d(o, x)+d(x, g^{-m}x)+d(x, g^m x)-d(g^{-m}o, g^m o))\\
&\le d(o, x)+4000\delta+\Delta\le \Lambda+4000\delta+\Delta.
\end{align*}
Recall that $\Lambda=NL(U)+4000\delta+1$. 
Moreover, $\Delta\le 2000\delta$ by Lemma \ref{lemma.quasiaxis}.
Thus, we have
\[
\Lambda+\Delta+4000\delta \le NL(U)+10^4\delta+1\le C_0(\delta, N)(L(U)+1).
\]
This implies \eqref{eq.qa}.

Next, we show \eqref{eq.wa}.
Since $g \in U^N$,
by Lemma \ref{lemma.disp}, we have $\tr{g} \le N L(U)$.
By Lemma \ref{lemma.tube}, 
for every $m \ge 0$,
we have
\begin{align*}
(g^m o, h^m o)_o 
&\le \Dc(\tr{g}+1)+10\Lambda+100\delta\\
&\le \Dc(N L(U)+1)+10(NL(U)+4000\delta+1)+100\delta\\
&\le (\Dc N+10N+10^5\delta+10)(L(U)+1).
\end{align*}
For every $m \ge 0$, we have
\[
(g^m o, h^m o)_o \le C_1(\delta, \Dc, N)(L(U)+1),
\]
where $C_1(\delta, \Dc, N)=\Dc N+10N+10^5\delta+10$.
The same inequality holds for $(g^m o, h^{-m}o)_o$.
Thus, we have \eqref{eq.wa}.

Finally, we show \eqref{eq.D} and \eqref{eq.Dup}.
Since $\tr{g}=\tr{g^{-1}}$,
by the assumption \eqref{assump.h}, for every $m \in \Z$, we have
\[
d(o, g^m o) \ge \tr{g^m}=|m|\tr{g} \ge |m|(L(U)+1).
\]
The analogous inequality holds for $d(o, h^m o)$ since $\tr{g}=\tr{h}$.
Hence we obtain \eqref{eq.D}.
Furthermore, 
we have
\[
d(o, g^m o) \le |m|((N+2)L(U)+1),
\]
by the triangle inequality, \eqref{eq.basepoint}, and Lemma \ref{lemma.disp} \eqref{item.Un} since $\{g, h\}\subset U^{N+2}$.
The analogous inequality holds for $d(o, h^m o)$.
Thus, we obtain \eqref{eq.Dup}.
\qed

\medskip

We now prove  Proposition \ref{prop.schottky} by constructing the required Schottky set.

\proof[Proof of Proposition \ref{prop.schottky}]
For all $n, m, r \in \N$,
let
\[
S_{n, m, r}:=\Big\{ (u_1\cdots u_n)^r, (u_1\cdots u_n)^{-r} \ : \ u_i \in \{g^m, h^m\}, i=1, \dots, n\Big\}.
\]
Note that $S_{n, m, r} \subset U^{(N+2)nmr}$ since $U=U^{-1}$.
We will show that $S_{n, m, r}$ is a required Schottky set at $o$ for appropriately chosen $n$, $m$, and $r$.
As we will see, these integers depend only on $\Dc$, $N$, and $\delta$,
and we will define 
$M=(N+2)nmr$.

By Lemma \ref{lemma.separation} \eqref{eq.qa} and \eqref{eq.wa},
we have a constant $C_1(\delta, \Dc, N)>0$ such that
for every $m \in \N$ and for all $s, s' \in \{g^m, g^{-m}, h^m, h^{-m}\}$ with $s\neq s'$,
\[
(s o, s' o)_o \le C_1(\delta, \Dc, N)(L(U)+1).
\]
Furthermore, by \eqref{eq.D}, we have
\[
\min\left\{d(o, g^{m} o), d(o, h^{m}o)\right\} \ge m(L(U)+1).
\]
Set
\[
m:=\cl{2C_1(\delta, \Dc, N)+4\delta+1}.
\]
For
$C_1:=C_1(\delta, \Dc, N)(L(U)+1)$ and $D_1:=m(L(U)+1)$,
we have
\begin{align*}
D_1=m(L(U)+1)
&\ge 2C_1(\delta, \Dc, N)(L(U)+1)+4\delta+L(U)+1\\
&=2C_1+4\delta+L(U)+1.
\end{align*}
In particular, we have $D_1>2C_1+2\delta$, hence,
by Lemma \ref{lemma.free}, $\{g^m, h^m\}$ freely generates a free subgroup in $\Isom(X)$.
Thus, we have
$\# S_{n, m, r}=2^{n+1}$.
Set
\[
n:=\cl{\log_2 100},
\]
so that $\# S_{n, m, r}\ge 200$.

We consider each $w \in S_{n, m, r}$ as a reduced word on $\{g^m, g^{-m}, h^m, h^{-m}\}$ of length $nr$.
For every pair of such reduced words
\[
w=s_1\cdots s_{nr} \quad \text{and} \quad w'=s_1'\cdots s_{nr}',  \quad \text{where $s_i, s_i'\in \{g^m, g^{-m}, h^m, h^{-m}\}$},
\]
let $k$ be the largest $i \le n$ such that $s_j=s_j'$ for all $j \in \{1, \dots, i\}$.
We note that each element in $S_{n, m, r}$ is the $r$-th power of a word of length $n$.
Hence, if $w\neq w'$ as words,
then $k<n$.
By definition of the Gromov product, 
we have
\begin{equation}\label{eq.gp_o}
(wo, w'o)_o \le d(o, s_1\cdots s_k o)+(wo, w'o)_{s_1\cdots s_k o}.
\end{equation}
Note that $d(o, s_1\cdots s_k o)\le (n-1)m((N+2)L(U)+1)$ by \eqref{eq.Dup} since $k< n$.
Moreover, since $D_1>2C_1+4\delta$, we have
\begin{equation}\label{eq.gp_k}
(wo, w'o)_{s_1\cdots s_k o} \le C_1+2\delta=C_1(\delta, \Dc, N)(L(U)+1)+2\delta,
\end{equation}
for a $(C_1, D_1)$-chain by Lemma \ref{lemma.chain} \eqref{lemma.chain3}.
Combining \eqref{eq.gp_o} with \eqref{eq.gp_k},
we obtain the following:
for all $w, w' \in S_{n, m, r}$ with $w\neq w'$,
\begin{align*}
(wo, w'o)_o 
\le C_2(\delta, \Dc, N) (L(U)+1),
\end{align*}
where $C_2(\delta, \Dc, N):=(n-1)m(N+2)+C_1(\delta, \Dc, N)+2\delta$.
Note that $C_2(\delta, \Dc, N)$ is a constant depending only on $\delta$, $\Dc$, and $N$.
Furthermore, $C_2(\delta, \Dc, N)>  3\delta$ since $m > 3\delta$.

For each $w \in S_{n, m, r}$, 
by Lemma \ref{lemma.chain} \eqref{lemma.chain2} applied to a $(C_1, D_1)$-chain of length $nr$,
we have
$d(o, wo) \ge nr(L(U)+1)$.
Set
\[
r:=\cl{40C_2(\delta, \Dc, N)+100\delta+1}.
\]
Then, since $n\ge 1$, we have
\[
nr (L(U)+1) \ge 40C_2(\delta, \Dc, N) (L(U)+1)+100\delta+L(U).
\]

Summarizing the above discussion, we obtain the following:
Let
\[
C_0:=C_2(\delta, \Dc, N) (L(U)+1) \quad \text{and} \quad D_0:=nr (L(U)+1).
\]
Then, we have
\[
\max_{w_1\neq w_2, w_1, w_2 \in S_{n, m, r}}(w_1 o, w_2 o)_o \le C_0 \quad \text{and} \quad \min_{w\in S_{n, m, r}}d(o, w o) \ge D_0,
\]
satisfying $D_0>2C_0+4\delta$.
Since $S_{n,m,r}$ is a finite symmetric set with $\# S_{n, m, r}\ge 200$,
Lemma \ref{lemma.criterion} implies that $S_{n, m, r}$ is a $(1/100, C_0+3\delta, D_0)$-Schottky set at o.
Furthermore, by construction, we have
\[
D_0\ge 40 C_0 +100\delta+L(U). 
\]
Letting
$C(\delta, \Dc, N):=2C_2(\delta, \Dc, N)$ and $D(\delta, \Dc, N):=nr$,
we define
\[
C:=C(\delta, \Dc, N)(L(U)+1) \quad \text{and} \quad D:=D(\delta, \Dc, N)(L(U)+1).
\]
Then, since $C=2C_0$ and $C_2(\delta, \Dc, N)\ge 3\delta$, we have $C_0+3\delta\le C$, and $D=D_0$.
Thus, $S_{n, m, r}$ is a $(1/100, C, D)$-Schottky set at $o$ satisfying that
\[
D \ge 20 C+100\delta+L(U).
\]

Finally, setting $M=M(\delta, \Dc, N):=(N+2)nmr$,
we have $S_{n, m, r}\subset U^M$.
As we have shown above, $\# S_{n, m, r}\ge 200$.
Therefore, the set $S:=S_{n, m, r}$ is a required Schottky set at $o$.
This completes the proof of Proposition \ref{prop.schottky}.
\qed

\section{Random walks on hyperbolic spaces}\label{sec.rw}

In this section, we obtain a lower bound for the drift of a random walk by using the pivotal time technique introduced in \cite{G}.
Our argument follows \cite[Section 4]{G}, but an essential point for our purpose
is to keep careful track of all constants and their dependence on the initial data. 
We quote two technical results from \cite[Section 4A]{G}, Propositions \ref{prop.y-p} and \ref{prop.pivot}, and then 
give a self-contained proof of the main result of this section, Proposition \ref{prop.drift}.

For $\delta\ge 0$, let $X$ be a $\delta$-hyperbolic space.
Let $C_0, D$ be positive constants.
Throughout this section,
we fix one  $\left(1/100, C_0, D\right)$-Schottky set $S$ at $o$ in $X$ with 
\[
D \ge 20 C_0+100\delta+1.
\]
An element from $S$ is called a Schottky element. 

\subsection{Pivotal times}\label{sec.pivot}

We recall the pivotal time construction from \cite{G}. 
The setting and  notation are the same as in \cite[Section 4A]{G}.

We start with the definition of chain-shadows.

\begin{definition}
Let $C\ge 0$ and $y, y^+, z \in X$.
The $C$-{\bf chain-shadow} $\mathcal{CS}_y(y^+; C)$ of $y^+$ seen from $y$ is the set of $z \in X$ such that
there exists a $(C, 2C+2\delta+1)$-chain 
$x_0=y, x_1, \dots, x_n=z$ 
and that
$(x_0, x_1)_{y^+}\le C$.
\end{definition}

Let $z_i$, $i\in \N_0$, be fixed isometries, and $a_i$, $b_i$ be Schottky elements from $S$.
Letting $s_i:=a_ib_i$ and $o \in X$,
we define the following sequences of points in $X$:
\[
y_i^-:=z_0 s_1 z_1\cdots s_{i-1}z_{i-1} o,
\quad 
y_i:=z_0s_1z_1\cdots z_{i-1}a_i o,
\quad \text{and}
\]
\[
y_i^+:=z_0s_1z_1 \cdots z_{i-1}a_i b_i o.
\]
We now define a sequence of {\bf pivotal times} $P_n \subseteq\{1, \dots, n\}$ inductively on $n$.
Let $P_0:=\emptyset$.
Suppose that $P_{n-1}$ is defined.
Let 
\[
k:=\max P_{n-1}.
\]
In the case when $P_{n-1}=\emptyset$, we define $k:=0$ and $y_k=o$.

We say that the {\bf local geodesic condition} is satisfied at time $n$ if
\[
(y_k, y_n)_{y_n^-} \le C_0, 
\quad
(y_n^-, y_n^+)_{y_n} \le C_0,
\quad
\text{and}
\quad
(y_n, y_{n+1}^-)_{y_n^+} \le C_0.
\]
If the local geodesic condition is satisfied at time $n$,
we define
\[
P_n:=P_{n-1}\cup\{n\}.
\]
Otherwise, we remove a ``backtrack'' from $P_n$ as follows:
First, consider the largest $m \in P_{n-1}$ such that
\[
y_{n+1}^- \in \mathcal{CS}_{y_m}(y_m^+, C_0+\delta).
\]
Then, we define
\[
P_n:=P_{n-1}\cap\{1, \dots, m\}.
\]
If there is no such $m$, then $P_n:=\emptyset$.
We finished the definition of $P_n$.

We recall the following key estimate, which gives a lower bound 
for $d(o, y_{n+1}^-)$. 
\begin{prop}[Proposition 4.5, \cite{G}]\label{prop.y-p}
For every $n \in \N_0$, 
we have 
\[
d(o, y_{n+1}^-)\ge \#P_n.
\]
\end{prop}

We briefly recall the main idea of the proof of Proposition \ref{prop.y-p}.
Let $k_1, \dots, k_p$ be the pivotal times at $n$.
Then, one shows that the sequence $o, y_{k_1}, y_{k_2}^-, y_{k_2}, \dots, y_{k_p}, y_{n+1}^-$ is a $(2C_0+4\delta, D-2C_0-3\delta)$-chain.
If $D \ge 20C_0+100 \delta+1$, then Lemma \ref{lemma.chain} (ii) is applied to this chain to show the claim.
See \cite[Proposition 4.5]{G} for details.

\medskip

\subsection{Exponential decay}\label{sec.exp}

Let $\mu_S$ be the uniform distribution on  the Schottky set $S$.
We consider independent random variables 
$a_i$, $b_i$, $i \in \N$, with the common distribution $\mu_S$.
Let 
\[
s_i:=a_ib_i.
\]
Note that $s_i$, $i\in \N$, are independent and distributed according to $\mu_S^{\ast 2}$.

We recall an important result from \cite{G}.
For every fixed sequence of isometries $z_0, z_1, \dots$ of $X$,
we have the following.

\begin{prop}[Proposition 4.10, \cite{G}]\label{prop.pivot}
There exists a universal constant $\kappa>0$ such that, for all $n \in \N_0$,
\[
\Prob(\#P_n \le \kappa n)\le e^{-\kappa n}.
\]
\end{prop}

For later use, we record a more explicit form of 
Proposition \ref{prop.pivot}.  

\begin{prop}\label{prop.kappa} 
There exists a universal constant $\kappa'>0$ such that, for all $n \in \N_0$,
\[
\Prob\left(\#P_n \le (1/2)n\right) \le e^{-\kappa' n}.
\]
\end{prop}

\begin{proof}
By \cite[Lemma 4.9]{G}, for all $i \in \Z$,
we have
\[
\Prob(\#P_{n+1}\ge i)\ge \Prob(\#P_n+U \ge i),
\]
where $U$ is a random variable independent from $\#P_n$ distributed as
\[
\Prob(U=-j)=9\left(\frac{1}{10}\right)^{j+1} \ \text{for $j>0$}, \
\Prob(U=0)=0 \ \text{and} \
\Prob(U=1)=\frac{9}{10}. 
\]
Thus, letting $U_i$, $i\in \N$, be independent copies of $U$, independently from the sequence $\# P_n$, $n\in \N_0$, for each $i\in \Z$ and each $n \in \N$, we have
\[
\Prob(\# P_n\ge i)\ge \Prob(U_1+\cdots+U_n \ge i).
\]
Therefore, for every $r>0$ and every $n \in \N_0$, we have
\[
\Prob(\# P_n \le r n) \le \Prob(U_1+\cdots+U_n \le r n).
\]
We show that for every $r\in (0, \E U)$, there exists $c_r>0$ such that
for all $n \in \N_0$,
\begin{equation}\label{eq:ldp}
\Prob(U_1+\cdots+U_n \le r n) \le e^{-c_r n}.
\end{equation}
This implies the claim by taking $r=1/2<\E U=71/90$.

Let us show \eqref{eq:ldp}.
For every $r \in (0, \E U)$, define $\e:=\E U-r>0$.
For every $t>0$, by the Markov inequality and the independence of $U_i$, $i\in \N$, we have the following:
\begin{align*}
\Prob\Big(U_1+\cdots+U_n \le rn\Big) 
&=\Prob\Big(-\sum_{i=1}^n (U_i-\E U) \ge \e n\Big)\\
&\le e^{-t \e n}\Big(\E e^{-t(U-\E U)}\Big)^n.
\end{align*}
Since $\E e^{-t (U-\E U)}=e^{o(t)}$,
for a small enough $t>0$ with $\e_t:=\e -o(t)/t>0$,
the last term in the above displayed inequality is at most $\exp(-t\e_t n)$ for all $n \in \N_0$.
Setting $c_r:=t \e_t>0$ yields \eqref{eq:ldp}.
This completes the proof.
\end{proof}

\subsection{Explicit lower bounds for drifts using Schottky sets}\label{sec.lower_drift}

Let $\mu$ be a countably supported probability measure on $\Isom (X)$.
Let $g_0, g_1, \dots$ be an i.i.d.\ sequence of random variables with the distribution $\mu$.
Recall that a $\mu$-random walk $(w_n)_{n\in \N_0}$ on $\Isom (X)$ starting from the identity $\id$ is defined by
\[
w_n:=g_0g_1\cdots g_{n-1} \quad \text{for $n \in \N$} \quad \text{and} \quad w_0:=\id.
\]
We define the underlying probability space as follows. 
Let $G$ be the countable semigroup in $\Isom (X)$ generated by the support of $\mu$.
We define $\Omega:=G^{\N_0}$ and endow $\Omega$ with the sigma-algebra generated by the cylinder sets.
Let $\Prob_\mu$ be the distribution on $\Omega$ defined as the pushforward of $\mu^{\otimes \N_0}$ by the map $(g_n)_{n \in \N_0} \mapsto (w_n)_{n\in \N_0}$.

The goal of this subsection is to prove the following proposition.
Recall that $\mu_S$ denotes the uniform distribution on a set $S$.

\begin{prop}\label{prop.drift}
Let $\delta, C_0, D\ge 0$, $Q\in \N$, and $\theta \in (0, 1]$.
Let $(X, d)$ be a $\delta$-hyperbolic space.
Fix an arbitrary $o \in X$.
Let $S \subset \Isom (X)$ be a $(1/100, C_0, D)$-Schottky set at $o \in X$ such that
$D \ge 20 C_0+100\delta+1$,
and $\mu$ be a countably supported probability measure  on $\Isom (X)$ with $\mu^{\ast Q} \ge \theta \mu_S^{\ast 2}$.

Then there exists a constant $c_\theta>0$ such that for all $n \ge 4Q/\theta$,
we have
\[
\Prob_\mu\left(\frac{1}{n}d(o, w_n o) \le \frac{\kappa\theta}{4Q} \right)\le 2 \exp\left(-\frac{c_\theta}{4Q} n\right).
\]
In particular,
for $\Prob_\mu$-almost every $(w_n)_{n \in \N_0} \in G^{\N_0}$,
we have
\[
\liminf_{n \to +\infty}\frac{1}{n}d(o, w_n o) \ge \frac{\kappa\theta}{4Q},
\]
where $\kappa$ is the universal constant in Proposition \ref{prop.pivot}.
\end{prop}

This proposition is a quantitative  version of the statement of \cite[Theorem 1.1]{G}. 
We track the constants in the proof to obtain an explicit, uniform exponential estimate, which in particular gives a lower bound for the drift. 
We give a self-contained proof to make the dependence on the constants 
clear, which are important for our argument.

In particular,  this allows us to relate the drift of a random walk to the weak acylindricity of the group action. 
Note that Proposition \ref{prop.drift} holds for all $(1/100, C_0, D)$-Schottky set $S$ at an arbitrary $o$ with $D\ge 20 C_0+100\delta+1$.
 In the next section, we apply the proposition to a carefully constructed Schottky set with a chosen basepoint $o$.

\medskip

Before proving Proposition  \ref{prop.drift},
we describe an alternative  realization of a $\mu$-random walk.
This realization will be used in the proof.

By assumption,  $S$ is a $(1/100, C, D)$-Schottky set at $o$ in $\Isom (X)$
such that  there exist $Q \in \N$ and $\theta \in (0, 1]$ with $\mu^{\ast Q}\ge \theta \mu_S^{\ast 2}$.
Hence, we have an auxiliary measure $\mu_{\rm aux}$ satisfying
\[
\mu^{\ast Q}=\theta \mu_S^{\ast 2}+(1-\theta)\mu_{\rm aux}.
\]
Let $(\e_i)_{i\in \N_0}$ be a sequence of i.i.d.\ Bernoulli random variables with 
\[
\Prob(\e_i=1)=\theta
\quad \text{and} \quad 
\Prob(\e_i=0)=1-\theta.
\]
Let $a_i$, $b_i$, $i \in \N_0$, be random variables distributed according to $\mu_S$,
and $h_i$, $i\in \N_0$, be random variables distributed according to $\mu_{\rm aux}$.
We assume that all random variables are independent and defined on a common probability space.
For each $i \in \N_0$,
let 
\[
\gamma_i:=
\begin{cases}
a_i b_i & \text{if $\e_i=1$},\\
h_i & \text{if $\e_i=0$}.
\end{cases}
\]
Note that $\gamma_i$, $i\in \N_0$, are independent and have the same distribution $\mu^{\ast Q}$.
Enlarging the probability space, we define all random variables $\e_i$, $a_i$, $b_i$, $h_i$, and $g_i$, $i\in \N_0$, on a common probability space so that
$\gamma_i=g_{Qi}\cdots g_{Q(i+1)-1}$ for each $i \in \N_0$ almost surely (cf.\ \cite[Lemma A1]{BP}).

Let $t_1, t_2, \cdots$, be the times $i$ when $\e_i=1$.
Fix $n \in \N_0$.
Let $\tau(n)$ be the largest $j \in \N$ such that $Q(t_j+1)\le n$ so that $[Q t_j, Q(t_j+1))\subset[0, n)$.
We define $\tau(n):=0$ if there is no such $j$.
Furthermore,
let
\[
s_j':=a_j'b_j' \quad \text{where $a_j':=a_{t_j}$ and $b_j':=b_{t_j}$ for $j \in \N$},
\]
\[
z_j:=g_{Q(t_j+1)}\cdots g_{Q t_{j+1}-1}
\quad \text{and} \quad
z'(n):=g_{Q(t_{\tau(n)}+1)}\cdots g_{n-1} \quad \text{for $j \in \N_0$}.
\]
In the above,
we understand $t_0:=-1$, and $z_0=\id$ if $t_1=0$.
This construction yields
\[
w_n o=z_0 s_1' z_1\cdots z_{\tau(n)-1} s_{\tau(n)}' z'(n) o.
\]

Associated to the sequence $z_0, s_1'=a_1'b_1', z_1, s_2'=a_2'b_2', \dots, z'(n)$,
we define the sequence of pivotal times 
\[
P_1, P_2, \dots, P_{\tau(n)-1}, P_{\tau(n)}^{(n)}.
\]
Note that $P_{\tau(n)}^{(n)}$ depends on $n$, so that we have added the superscript $(n)$ to $P_{\tau(n)}$ to indicate its dependency.
(For $j<\tau(n)$, the construction of $P_j$ does not involve $n$.)

We start the proof. 
\proof[Proof of Proposition \ref{prop.drift}]

Recall that $t_1, t_2, \dots$ are the times $i$ when $\e_i=1$ for $i\in \N_0$.
Since $\e_i$ are i.i.d.\ $(\{0, 1\}, \theta)$-Bernoulli random variables,
the random variables $t_{j+1}-t_j$, $j\in \N_0$, where $t_0=-1$, are geometric with parameter $\theta$ and independent.
Therefore, 
a direct estimate as in the proof of Proposition \ref{prop.kappa}
yields the following:
For every $\e>0$, there exists a constant $I_{\e, \theta}>0$ such that for every $j\in \N$,
we have
\[
\Prob\Bigg(t_j-t_0\ge \Big(\frac{1}{\theta}+\e\Big) j\Bigg)=\Prob\Bigg(\sum_{i=0}^{j-1}(t_{i+1}-t_i)\ge \Big(\frac{1}{\theta}+\e\Big) j\Bigg) \le e^{-I_{\e, \theta} j}.
\]

Fix $\e:=1/\theta$. 
Let $\beta:=\theta/(4Q)$.
Then, we have
\[
t_{\cl{\beta n}}+1 \le \frac{2}{\theta} \cl{\beta n} \implies Q(t_{\cl{\beta n}}+1) \le \frac{2Q}{\theta} \cl{\beta n} \le \frac{2Q}{\theta} (\beta n+1) \le n,
\]
for all $n \ge 4Q/\theta$,
whence $\tau(n) \ge \cl{\beta n}$.
Letting $I_{\theta}:=I_{\e, \theta}$ with $\e=1/\theta$,
we have, for every $n \in \N_0$,
\begin{equation}\label{eq:tau}
\Prob(\tau(n)<\cl{\beta n})\le e^{-I_\theta \cl{\beta n}}.
\end{equation}

We first prove the following:
there exists a constant $c_\theta>0$ such that
for all $n \ge 4Q/\theta$,
we have
\begin{equation}\label{eq.prop.un}
\Prob\left(\# P_{\tau(n)}^{(n)} \le \frac{\kappa \theta}{4Q} n \right) \le 2\exp\left(-\frac{c_\theta}{4Q} n\right).
\end{equation}

To show \eqref{eq.prop.un}, we use Proposition \ref{prop.pivot}.
This proposition applies once $\tau(n)$, $(t_i)_{i \le \tau(n)}$, $(z_j)_{j<\tau(n)}$, and $z'(n)$ are fixed.
Conditioned on $(\e_i)_{i \in \N_0}$ and $g_i$ for $i \notin [Q t_j, Q(t_j+1))$ where $j\in \N$,
these random variables are fixed and we are in the setting to apply the proposition.
Let $\mathcal{E}$ denote the sigma-algebra generated by 
\[
\e_i \quad \text{for all $i \in \N_0$}, \quad \text{and} \quad g_i \quad \text{for all $i\notin \bigcup_{j=1}^{\tau(n)} \big[Q t_j, Q(t_j+1)\big)$}.
\]

Let $u_n:=\# P_{\tau(n)}^{(n)}$.
By Proposition \ref{prop.pivot},
there exists a universal constant $\kappa>0$ such that
for all $n \in \N_0$, almost surely
\[
\Prob\left(u_n \le \kappa \tau(n) \mid \mathcal{E}\right) \le e^{-\kappa \tau(n)}.
\]
Restricting on the event $\tau(n)\ge \cl{\beta n}$, 
we have almost surely
\begin{equation*}
\Prob\left(u_n \le \kappa \cl{\beta n} \mid \mathcal{E}\right)\le \Prob\left(u_n \le \kappa \tau(n) \mid \mathcal{E}\right) \le e^{-\kappa \tau(n)} \le e^{-\kappa \cl{\beta n}}.
\end{equation*}
Thus, for every $n \in \N_0$, we have
\begin{equation}\label{eq:un}
\Prob\left(u_n \le \kappa \cl{\beta n}, \tau(n) \ge \cl{\beta n}\right) \le e^{-\kappa\cl{\beta n}}.
\end{equation}
Therefore, combining \eqref{eq:tau} and \eqref{eq:un}, 
and  using $\beta=\theta/(4Q)$, we obtain 
\begin{align*}
\Prob\Big(u_n \le \kappa \cl{\beta n}\Big)
&\le \Prob\Big(\tau(n)<\cl{\beta n}\Big)+\Prob\Big(u_n \le \kappa \cl{\beta n}, \tau(n)\ge \cl{\beta n}\Big)\\
&\le e^{-I_\theta \cl{\beta n}}+e^{-\kappa \cl{\beta n}} \le 2 \exp\left(-\frac{c_\theta}{4Q}n\right),
\end{align*}
where 
$$c_\theta:=\theta\min\{I_\theta, \kappa\}>0.$$
Note that $c_\theta$ depends only on $\theta$ since $I_\theta$ depends only on $\theta$, and $\kappa$ is a universal constant.
This shows \eqref{eq.prop.un}.

Applying Proposition \ref{prop.y-p} with $w_n o=y_{\tau(n)+1}^-$,
for every $n \in \N_0$, we have almost surely 
\[
d(o, w_n o) \ge \#P_{\tau(n)}^{(n)}.
\]
By \eqref{eq.prop.un},
there exists a constant $c_\theta>0$ such that for all $n \ge 4Q/\theta$,
we have
\[
\Prob_\mu\left(d(o, w_n o) \le \frac{\kappa \theta}{4Q} n \right)\le \Prob\left(\#P_{\tau(n)}^{(n)} \le \frac{\kappa \theta}{4Q} n \right) \le 2 \exp\left(-\frac{c_\theta}{4Q} n\right).
\]
This shows the desired inequality. 

The final assertion follows from this inequality by the Borel-Cantelli lemma.
\qed

\section{Main results: Uniform lower bounds for drifts}\label{sec.drift}

Let $(X, \delta)$ be a $\delta$-hyperbolic space.
Suppose that $\mu$ is a probability measure on $\Isom(X)$ with a finite first moment with respect to $(X, d)$,
i.e., 
\[
\sum_{g \in \supp \mu} \mu(g)d(o, g o)<+\infty.
\]
Here $\supp \mu$ denotes the support of $\mu$.
The condition that $\mu$ has a finite first moment does not depend on the choice of $o$ by the triangle inequality.

We consider a $\mu$-random walk $(w_n)_{n \in \N_0}$ on $\Isom(X)$.
Recall that the drift $\ell(\mu)$ is defined by
$\ell(\mu)=\lim_{n \to+\infty}(1/n)\E d(o, w_n o)$.
Note that $\ell(\mu)$ does not depend on the choice of $o \in X$ by the triangle inequality.
By the Kingman subadditive ergodic theorem,
the following limit exists $\Prob_\mu$-almost surely and coincides with $\ell(\mu)$,
i.e., for $\Prob_\mu$-almost every $(w_n)_{n \in \N_0}$,
\[
\ell(\mu)=\lim_{n \to +\infty}\frac{1}{n}d(o, w_n o).
\]

The following lemma will be used to reduce the general case to the case when $1 \in \supp \mu$
to obtain a lower bound on $\ell(\mu)$.
Let $\delta_1$ denote the probability measure supported at the identity element $1$.

\begin{lemma}\label{lemma.lazy}
Let $\mu$ be a countably supported probability measure of finite first moment on $\Isom(X)$. Set
$\widetilde \mu:=(\mu+\delta_1)/2$.
Then,
we have
$\ell(\widetilde \mu)=\ell(\mu)/2$.
\end{lemma}

\proof
Let $(\chi_n)_{n \in \N_0}$ be an i.i.d.\ $(\{0, 1\}, 1/2)$-Bernoulli random variables,
and $t_n:=\chi_1+\cdots+\chi_n$.
Then, a $\widetilde \mu$-random walk $(\widetilde w_n)_{n \in \N_0}$ has the same distribution as that of $(w_{t_n})_{n \in \N_0}$ where $(w_n)_{n \in \N_0}$ is a $\mu$-random walk.
Since $t_n/n \to 1/2$, as $n \to \infty$, almost surely by the strong law of large numbers,
we have $\ell(\widetilde \mu)=\ell(\mu)/2$ by definition of the drift.
\qed

\subsection{Main theorem}

We now state our main theorem, from which Theorem \ref{thm.intro_meta} follows. 
Here $K_0$ denotes the absolute constant in Lemma \ref{lemma.independent.2}.

\begin{thm}\label{thm.meta}
Let $\delta, \Dc \ge 0$, $N, q \in \N$, and $\theta_0\in (0, 1]$.
Then there exist constants $\kappa_\star(\delta, \Dc, N, q, \theta_0)>0$ and $c_\star(\delta, \Dc, N, q, \theta_0)>0$ such that the following holds.

Let $(X, d)$ be a $\delta$-hyperbolic geodesic space.
Let $U$ be a finite symmetric set in $\Isom (X)$ with $1 \in U$. 
Assume 
$L(U^N) \ge K_0 \delta+1$. Assume that the action of $\<U\>$  on $X$
is  non-elementary and $\Dc$-weakly acylindrical.

Let $\mu$ be a countably supported probability measure on $\Isom (X)$
with 
$$\min_{u\in U}\mu^{\ast q}(u)\ge \theta_0.$$

Then, for all large enough $n$,
we have
\[
\Prob_\mu\Bigg(\frac{1}{n L(U)}d(o, w_n o) \le \kappa_\star(\delta, \Dc, N, q, \theta_0) \Bigg)\le  \exp\Bigg(-c_\star(\delta, \Dc, N, q, \theta_0) n\Bigg).
\]
Consequently, if $\mu$ has a finite first moment, then
\[
\ell(\mu)\ge \kappa_\star(\delta, \Dc, N, q, \theta_0) L(U).
\]
Moreover, there exists $M\in \N$ depending only on $\delta$, $\Dc$, and $N$, such that
\[
\kappa_\star(\delta, \Dc, N, q, \theta_0)=\frac{\theta_0^{2M}}{16qM}.
\]
\end{thm}

\proof
To apply Proposition \ref{prop.schottky}, we first check its assumptions. 
First, since $L(U^N) \ge K_0\delta+1$, Lemma \ref{lemma.independent.2} implies that
there exists a hyperbolic element $g \in U^{2N}$ such that
\[
\tr{g}\ge L(U)+4000\delta+1.
\]
Second, since the action of $\<U\>$ on $X$ is non-elementary,  $\<U\>$ does not preserve the set of fixed points of $g$ in $\partial X$.
Lastly, since the action of $\<U\>$ is $\Dc$-weakly acylindrical, 
we apply Proposition \ref{prop.schottky} and  obtain the following: 

There exists an integer $M=M(\delta, \Dc, 2N)$ and a point $o\in X$ such that $U^M$ contains a $(1/100, C, D)$-Schottky set $S$ at $o$
with
$D\ge 20C+100\delta+L(U)$, where
$C=C(\delta, \Dc, 2N)(L(U)+1)$ and $D=D(\delta, \Dc, 2N)(L(U)+1)$.

We will apply Proposition \ref{prop.drift}.
Since $U^{2M} \supset S^2$, and by assumption  $\min_{u \in U}\mu^{\ast q}(u) \ge \theta_0$,
we have
\[
\mu^{\ast (2Mq)} \ge \theta_0^{2M}\mb{1}_{S^2} \ge \theta_0^{2M}\mu_S^{\ast 2}.
\]
Here $\mb{1}_E$ denotes the indicator function on a set $E$ and $\mu_S$ denotes the uniform distribution on $S$.

Let $Q:=2Mq$ and $\theta:=\theta_0^{2M}$, so that $\mu^{\ast Q} \ge \theta \mu_S^{\ast 2}$.
Let $R:=L(U)>0$.
After rescaling the metric by $R^{-1}$,  the metric space $(X, R^{-1}d)$ is $(\delta/R)$-hyperbolic.
With respect to the rescaled metric, the set $S$   is 
a $(1/100, C/R, D/R)$-Schottky set  at $o$ on $(X, R^{-1}d)$, where 
\[
D/R\ge 20 (C/R)+100(\delta/R)+1.
\]
Applying Proposition \ref{prop.drift} to this setting, 
there exists a constant $c_\theta>0$ such that for all $n \ge 4Q/\theta$,
we have
\[
\Prob_\mu\left(\frac{1}{nL(U)}d(o, w_n o) \le \frac{\kappa\theta}{4Q} \right)\le 2 \exp\left(-\frac{c_\theta}{4Q} n\right),
\]
where $\kappa$ is the universal constant in Proposition \ref{prop.pivot}.
Letting 
\[
\kappa_\star(\delta, \Dc, N, q, \theta_0):=\frac{\kappa\theta}{4Q}=\frac{\kappa \theta_0^{2M}}{8Mq}
\quad \text{and} \quad
c_\star(\delta, \Dc, N, q, \theta_0):=\frac{c_\theta}{8Q}=\frac{c_\theta}{16Mq}
\]
we obtain the first claim.

Furthermore, 
for $\Prob_\mu$-almost every $(w_n)_{n \in \N_0}$,
\[
\liminf_{n\to +\infty}\frac{1}{Rn}d(o, w_n o) \ge \kappa_\star(\delta, \Dc, N, q, \theta_0).
\]

If $\mu$ has a finite first moment, the Kingman subadditive ergodic theorem 
implies that the above liminf is in fact a limit, and is equal to $\ell(\mu)/R$ for $\Prob_\mu$-almost every $(w_n)_{n \in \N_0}$.

Finally, we obtain the explicit value of the constant $\kappa_\star(\delta, \Dc, N, q, \theta_0)$ by  taking $\kappa=1/2$ as allowed by Proposition \ref{prop.kappa}.
\qed

\medskip

We now pass from drift to translation length. 
Theorem \ref{thm.meta} on the displacement $d(o, w_no)$ gives the corresponding result on the translation length $\tr{w_n}$.
We say a probability measure $\mu$ on $\Isom (X)$  is {\bf bounded} if for some (equivalently, for all) $x \in X$,
we have
$\sup_{g \in \supp \mu}d(x, g x)<+\infty$.

Let $X$ be a $\delta$-hyperbolic geodesic space.
Let $\mu$ be a countably supported, non-elementary, and bounded probability measure on $\Isom(X)$.
Then, we have
\[
\lim_{n \to +\infty}\frac{1}{n}\tr{w_n}=\ell(\mu) \quad \text{for $\Prob_\mu$-almost every $(w_n)_{n \in \N_0}$},
\]
\cite[Theorem 4.1]{MTwpd}.
For $\tr{w_n}$,
an exponential decay estimate was shown in \cite[Theorem 1.4]{MT},
and a more refined large deviation principle was established in \cite[Theorem 1.3]{BMSS}.

The following corollary is a uniform version of \cite[Theorem 1.4]{MT} under the additional assumption that a group action is weakly acylindrical.

\begin{cor}\label{cor.tr}
In addition to the assumptions of Theorem \ref{thm.meta},
we assume that $\mu$ is bounded.
Then, 
for sufficiently large  $n$, 
we have
\[
\Prob_\mu\left(\frac{1}{nL(U)}\tr{w_n} \le \frac{1}{2}\kappa_\star(\delta, \Dc, N, q, \theta_0)\right) \le  \exp\left(-\frac{1}{2}c_\star(\delta, \Dc, N, q, \theta_0)n\right).
\]
Here, the constants $\kappa_\star(\delta, \Dc, N, q, \theta_0)$ and $c_\star(\delta, \Dc, N, q, \theta_0)$ are the same as in Theorem \ref{thm.meta}.
\end{cor}

We use the following lemma to pass from displacement estimates to translation-length estimates. 
Let $o \in X$, and 
assume that $\mu$ is bounded.

\begin{lemma}[Lemma 10.4, \cite{BMSS}]\label{lemma.BMSS}
For any $\zeta>0$ and any $\e>0$,
there is an integer $N \in \N_0$ such that for any $n \ge N$, 
we have
\[
\Prob_\mu\left(\tr{w_n} \le \zeta n\right) \le (n+1)\Prob_\mu\left(d(o, w_n o)\le (\zeta+\e)n\right).
\]
\end{lemma}

In \cite{BMSS}, the lemma is stated for $\inf_{x \in X}d(x, w_n x)$ in place of $\tr{w_n}$.
It implies the present formulation since those two quantities differ by a uniformly bounded amount depending only on the hyperbolicity constant of $X$ (cf.\ Lemma \ref{lemma.tr}).

\proof[Proof of Corollary \ref{cor.tr}]
Letting $\zeta=\e=(1/2)\kappa_\star(\delta, \Dc, N, q, \theta_0)$,
we apply Lemma \ref{lemma.BMSS} to the present setting where the metric is rescaled by $1/L(U)$:
there exists an integer $N$ such that for all $n \ge N$, 
\begin{align*}
\Prob_\mu\left(\frac{1}{nL(U)}\tr{w_n} \le \zeta\right)
\le (n+1)\Prob_\mu\left(\frac{1}{nL(U)}d(o, w_n o)\le 2 \zeta\right).
\end{align*}
Furthermore, by Theorem \ref{thm.meta}, the right hand side is bounded from above by
\[
(n+1) \exp\left(-c n\right)=\exp\left(-cn+\log(n+1)\right)
\le 
\exp\left(-\frac{1}{2}c n\right)
\]
for sufficiently large $n$, where $c=c_\star(\delta, \Dc, N, q, \theta_0)>0$.
\qed

\subsection{Uniform bounds over actions}\label{sec.main.action}

We now state our main results.
They are obtained by applying Theorem \ref{thm.meta} in two different settings. One is treated in this section, and the other is in the next section.

Let $(X, d)$ be a $\delta$-hyperbolic geodesic space. 
Let $\Gamma$ be a finitely generated group.
We consider an isometric action $\rho$ of $\Gamma$ on $X$, equivalently, a group-homomorphism 
\[
\rho:\Gamma \to \Isom (X).
\]
Suppose that  $\rho:\Gamma \to \Isom (X)$ gives a faithful (i.e., $\rho$ is injective), non-elementary, isometric action.

Recall that a measure $\nu$ on $\Gamma$ is called admissible if the support of $\nu$ generates $\Gamma$ as a semigroup.
Let $\nu$ be an admissible probability measure on $\Gamma$ with a finite first moment relative to some (equivalently, all) word metric.

We consider a $\mu$-random walk $(w_n)_{n \in \N_0}$ on $\Isom (X)$
where $\mu=\rho_\ast \nu$.
Note that $\mu$ has a finite first moment on $(X, d)$.
Indeed, this follows from
\[
d(o, \rho(\gamma)o)
\le |\gamma|_\Sigma \max_{s \in \Sigma}d(o, \rho(s)o),
\]
for every $\gamma \in \Gamma$, where $|\cdot|_\Sigma$ denotes the word norm relative to $\Sigma$.
We define the associated drift by
\[
\ell(\nu, \rho):=\ell(\rho_\ast \nu).
\]

It is known that $\ell(\nu,\rho) >0$ as a special case of \cite[Theorem 1.2]{MT} (see also \cite[Theorem 1.1]{G} for an alternative proof).
We will give a lower bound  on $\ell(\nu,\rho)$. 
For this purpose, we fix a finite generating set $\Sigma$ with $1 \in \Sigma$.
For each action $\rho$, the quantity  $L(\rho(\Sigma))$ gives 
a natural scale. Our first theorem gives a uniform linear lower 
bound on the drift $\ell(\nu,\rho) >0$ in terms of $L(\rho(\Sigma))$.
One key point is the multiplicative 
constant in the lower bound
below is independent of the action $\rho$.

\begin{thm}[Main theorem I]\label{thm.drift_general}
Let $\delta, \Dc \ge 0$, and $N \in \N$.
Let $(X, d)$ be a $\delta$-hyperbolic geodesic space. 
Let $\Gamma$ be a finitely generated group, and let $\nu$ be an admissible probability measure on $\Gamma$ with finite first moment.
Let $\Sigma$ be a finite symmetric generating set of $\Gamma$ with $1 \in \Sigma$.

Then, there exists a constant  $\kappa(\delta, \Dc, N, \Gamma, \Sigma, \nu)>0$ such that for
every faithful, non-elementary and $\Dc$-weakly acylindrical isometric action $\rho:\Gamma \to \Isom (X)$  satisfying $L(\rho(\Sigma)^N)\ge K_0\delta+1$,
we have
\[
\ell(\nu, \rho)\ge \kappa(\delta, \Dc, N, \Gamma, \Sigma, \nu)L(\rho(\Sigma)).
\]
Here $K_0$ is the absolute constant in Theorem \ref{thm.meta}.
\end{thm}

\proof
Let $U:=\rho(\Sigma)$, and $\mu:=\rho_\ast \nu$.
We may assume that  $1 \in \supp \nu$ since otherwise Lemma \ref{lemma.lazy}
reduces the general case to this one at the cost of a factor of $2$.  
To apply Theorem \ref{thm.meta}, we first verify its assumptions and identify the relevant constants. 
We claim that there exist $q=q(\Gamma, \Sigma, \nu) \in \N$ and $\theta_0=\theta_0(\Gamma, \Sigma, \nu) \in (0, 1]$ such that for every faithful isometric action $\rho$ of $\Gamma$ on $X$,
we have
$\min_{u \in \rho(\Sigma)}(\rho_\ast \nu)^{\ast q} \ge \theta_0$.

Indeed, since $1 \in \supp \nu$ and $\nu$ is admissible, the support of $\nu^{\ast n}$ is weakly increasing in $n$, and the union of these supports coincides with $\Gamma$.
Moreover, since $\Sigma$ is a finite set,
there exists $q \in \N$ such that $\Sigma \subset \supp \nu^{\ast q}$.
Letting $\theta_0:=\min\{\nu^{\ast q}(s) \ : \ s \in \Sigma\}>0$,
we have $\min_{s\in \Sigma}\nu^{\ast q}(s)\ge \theta_0$.
Furthermore, by the injectivity of $\rho$,
we have $\min_{u\in \rho(\Sigma)}(\rho_\ast \nu)^{\ast q}(u)\ge \theta_0$, as required.
By definition, $q$ and $\theta_0$ depend only on $\Gamma$, $\Sigma$, and $\nu$, but not on $\rho$.
Moreover, we have $\rho(\Sigma)^{-1}=\rho(\Sigma)$ since $\Sigma$ is symmetric.

Note that $\mu=\rho_\ast \nu$ has a finite first moment on $(X, d)$ for every $\rho$.
Therefore, by Theorem \ref{thm.meta},
letting $\kappa(\delta, \Dc, N, \Gamma, \Sigma, \nu):=\kappa_\star(\delta, \Dc, N, q, \theta_0)$,
we obtain
\[
\ell(\nu, \rho)=\ell(\rho_\ast \nu)\ge \kappa(\delta, \Dc, N, \Gamma, \Sigma, \nu) L(\rho(\Sigma)).
\]
This shows the theorem.
\qed

\medskip

Hadamard manifolds provide an important class of examples
to which Theorem  \ref{thm.drift_general} applies. 
In this setting, some of the assumptions are automatically satisfied.

\begin{cor}\label{cor.drift_Hn}
Fix an integer $n \ge 2$ and a real $a \in [1, +\infty)$.
Let $\Gamma$ be a finitely generated group, $\Sigma$ a finite symmetric generating set of $\Gamma$ with $1 \in \Sigma$, 
and 
$\nu$ an admissible, probability measure on $\Gamma$ with a finite first moment.

Then, there exists a constant $\kappa(n, a, \Gamma, \Sigma, \nu)>0$ such that the following holds.

Let $X$ be an $n$-dimensional Hadamard manifold with sectional curvature $K_X \in [-a^2, -1]$. 
For every faithful, isometric action $\rho:\Gamma \to \Isom (X)$ such that $\rho(\Gamma)$ is discrete and non-virtually-nilpotent,
we have
\[
\ell(\nu, \rho)\ge \kappa(n, a, \Gamma, \Sigma, \nu)L(\rho(\Sigma)).
\]
\end{cor}

The above result may be viewed as a substantial extension of \cite[Theorem 3.3]{AGG}, 
which concerns the special case where
$X=\H^2$ has  constant curvature $-1$, $\Gamma$ is the fundamental group of a closed orientable surface of genus at
least two, and $\rho$ is cocompact.

\proof[Proof of Corollary \ref{cor.drift_Hn}]
By Lemma \ref{lemma.independent},
there exists $N=N(n, a)$ such that $L(\rho(\Sigma)^N)\ge K_0\delta +1$,
where $K_0$ is the absolute constant in Theorem \ref{thm.meta}.
Also, the faithful action $\rho$ of $\Gamma$ on $X$ is $\Dc$-weakly acylindrical with $\Dc=\Dc(n, a)$ by Corollary \ref{cor.Hn}, and   non-elementary by Lemma \ref{lemma.nonelem}.
Note that the hyperbolicity constant $\delta$ is at most that of $\H^2$ (see Section \ref{sec.hadamard}).
Thus, Theorem \ref{thm.drift_general} yields the claim with $\kappa(n, a, \Gamma, \Sigma, \nu):=\kappa(\delta, \Dc, N, \Gamma, \Sigma, \nu)$.
\qed

\subsection{Uniform bounds over finitely generated subgroups}\label{sec.main.subgr}

We now treat the second of the two settings mentioned above. 
In this setting, 
instead of 
fixing  a group $\Gamma$ and a generating set $\Sigma$, 
we consider simple random walks 
associated to finite sets $U$ in $\Isom(X)$ of fixed cardinality. 
This formulation will be used in Sections \ref{sec.hyp.gp} and \ref{sec.MCG}.

For a finite set $U$, we denote by $\mu_U$ the uniform distribution on $U$.

\begin{thm}[Main theorem II]\label{thm.drift2}
For every $\delta, \Dc \ge 0$, $N, m \in \N$,
there exists a constant $\kappa(\delta, \Dc, N, m)>0$ such that the following holds:

Let $(X, d)$ be a $\delta$-hyperbolic geodesic space.
For every finite symmetric set $U$ in $\Isom (X)$ with $\#U=m$ such that $\<U\>$ is non-elementary, $\Dc$-weakly acylindrical, 
and $L(U^N)\ge  K_0\delta+1$,
we have
\[
\ell(\mu_U)\ge \kappa(\delta, \Dc, N, m)L(U).
\]
Here $K_0$ is the absolute constant in Theorem \ref{thm.meta}.

In fact, we have
\[
\kappa(\delta, \Dc, N, m)=\frac{1}{16M (2m)^{2M}},
\]
where $M$ is a positive integer depending only on $\delta$, $\Dc$, and $N$.
\end{thm}

\proof
To apply Theorem \ref{thm.meta} we verify its assumption and identify the relevant constants. 
We may assume that  $1 \in \supp \mu_U$ since otherwise we use Lemma \ref{lemma.lazy}
at the cost of replacing $m$ by $2m$ in below. 
By definition, we have $\min_{u \in U}\mu_U(u)=1/m$ and thus $q=1$ and $\theta_0=1/m$.
Applying Theorem \ref{thm.meta},
we obtain the claim with the constant $\kappa(\delta, \Dc, N, m):=\kappa_\star(\delta, \Dc, N, q, \theta_0)$.
Furthermore, we have
\[
\kappa(\delta, \Dc, N, m)=\frac{\theta_0^{2M}}{16qM}=\frac{1}{16M m^{2M}},
\]
where $M=M(\delta, \Dc, N)$.
This shows the claim.
\qed

\medskip

We now give an important application of Theorem \ref{thm.drift2} to pinched Hadamard manifolds.
In this setting, some of the assumptions are automatically satisfied.

\begin{cor}\label{cor.srwHn}
Let $n \in \N$, $a\in [1, +\infty)$, and $m \in \N$.
Then, there exists a constant $\kappa(n, a, m)>0$ such that the following holds:

Let $X$ be an $n$-dimensional Hadamard manifold with sectional curvature $K_X \in [-a^2, -1]$.
For every finite symmetric set $U$ in $\Isom(X)$ with $\# U=m$, such that $\<U\>$ is discrete and non-virtually-nilpotent, 
we have
\[
\ell(\mu_U) \ge \kappa(n, a, m)L(U).
\]

Moreover, we have
\[
\kappa(n, a, m) = \frac{1}{16M (2m)^{2M}},
\]
where $M$ is a positive integer depending only on $n$ and $a$.
\end{cor}

\proof
By Lemma \ref{lemma.independent}, there exists $N=N(n, a)\in \N$ such that $L(U^N)\ge K_0\delta+1$ where $K_0$ is the absolute constant in Theorem \ref{thm.meta} and $\delta$ denotes the hyperbolicity constant of $\H^2$.
Furthermore, 
the action of $\<U\>$ on $X$ is $\Dc$-weakly acylindrical where $\Dc=\Dc(n, a)$ by Corollary \ref{cor.Hn}, and
non-elementary by Lemma \ref{lemma.nonelem}.
Therefore, Theorem \ref{thm.drift2} implies the claim.
\qed

\begin{question}\label{question.constant}
In Corollary \ref{cor.srwHn},
can one take the constants $\kappa(n, a, m)$ so that $\liminf_{m \to +\infty}\kappa(n, a, m)>0$ with both $n$ and $a$ fixed? 
\end{question}

As another consequence, we obtain the following drift dichotomy for 
 finitely generated discrete subgroup $\<U\>$ of isometries of pinched Hadamard manifolds with $U=U^{-1}$.
\begin{cor}\label{cor.dichotomy}
Let $n \in \N$, $a\in [1, +\infty)$, and $m \in \N$.
Then, there exists a constant $c(n, a, m)>0$ such that the following holds:

Let $X$ be an $n$-dimensional Hadamard manifold with sectional curvature $K_X \in [-a^2, -1]$.
For every finite set $U$ in $\Isom(X)$ with $\# U=m$ and $U=U^{-1}$ such that $\<U\>$ is discrete,
exactly one of the following holds: 
\begin{enumerate}[(1)]
\item $\ell(\mu_U)=0$ if $\<U\>$ is virtually nilpotent. 
\item $\ell(\mu_U) \ge c(n, a, m)$ if $\<U\>$ is non-virtually-nilpotent.
\end{enumerate}
\end{cor}

\proof
Set 
$$c(n, a, m):=\e(n, a) \kappa(n, a, m),$$ where $\e(n, a)$ is a Margulis constant of $X$ and $\kappa(n, a, m)$ is the constant in Corollary \ref{cor.srwHn}. 

If $\<U\>$ is discrete and non-virtually-nilpotent,
then by the Margulis lemma, we have $L(U) \ge \e(n, a)$ \cite[Section 8.3]{BGS}.
Thus, Corollary \ref{cor.srwHn} implies that 
\[
\ell(\mu_U) \ge \e(n, a) \kappa(n, a, m)=c(n, a, m).
\]

On the other hand, if $\<U\>$ is virtually nilpotent, 
then the Poisson boundary of $(\<U\>, \mu_U)$ is trivial (cf.\ \cite[Proposition 1.4]{KV}).
Note that $\mu_U$ is symmetric since $U=U^{-1}$.
Considering the word norm $|\cdot\,|_U$ on $\<U\>$ with respect to $U$,
the corresponding drift $\ell_U(\mu_U)$ is equal to $0$ by \cite[Corollary 2]{KL}.
Since $d(o, g o) \le (\max_{u \in U}d(o, uo))|g|_U$ for all $g \in \<U\>$,
we have $\ell(\mu_U)=0$.
\qed

\begin{remark}\label{remark.dichotomy}
The conclusion  of Corollary \ref{cor.dichotomy} may fail if we do not assume $U= U^{-1}$.
Indeed, take $U$ to consist of  a single element generating an infinite cyclic group.

Also, we can not remove the assumption that $\<U\>$ is discrete.
Indeed, in  $\H^n$, there exists a finitely generated non-virtually-nilpotent solvable group $\<U\>$ in a parabolic subgroup in $\Isom(\H^n)$ such that
$\< U\>$ is not discrete and  $\ell(\mu_U)=0$.
\end{remark}

\begin{question}\label{question.gap}
In Corollary \ref{cor.dichotomy},
can the constants $c(n, a, m)$ be chosen so that $\liminf_{m \to +\infty}c(n, a, m)>0$ for each fixed $n$ and $a$? 
\end{question}

\section{Applications to drift spectra}\label{sec.applications}
We apply our results to the spectra of drifts in several settings. 
\subsection{Hyperbolic groups}\label{sec.hyp.gp}
We first discuss the case of hyperbolic groups. 
 Let $G$ be a non-elementary hyperbolic group, and 
 $U_0$ a finite symmetric generating set.
Let $\Cay(G,U_0)$ be the Cayley graph with the induced graph metric where each edge has a unit length.
The Cayley graph $\Cay(G, U_0)$ is $\delta$-hyperbolic for some $\delta \ge 0$.
We fix such $U_0$ once and for all. With respect to the metric in $\Cay(G, U_0)$,
we define the joint minimal displacement $L(U)$ for $U$ and the drift $\ell(\mu)$ for a $\mu$-random walk on $G$.

We consider the left action of $G$ on $\Cay(G,U_0)$.  
This action is properly discontinuous and cocompact, hence it is acylindrical. Therefore the action is $\mathcal D$-weakly acylindrical for 
some $\mathcal D$. The constant $\mathcal D$ depends only on $\# U_0$ and $\delta$.

Below $K_0$ denotes the absolute constant in Theorem \ref{thm.meta}.

\begin{cor}[Linear estimate in Hyperbolic group]\label{cor.drift.hyp.gp}
Fix  $\delta \ge 0$ and positive integers  $m_0$, and $m$.
There exists a constant $\kappa_1(\delta, m_0, m)>0$ such that the 
following holds. 

Let $G$ be a hyperbolic group, and $U_0$ a finite symmetric generating
set of $G$ such that $\# U_0=m_0$ and that $\Cay(G,U_0)$ is $\delta$-hyperbolic.
For every finite symmetric set $U$ in $G$ with $\# U=m$ such that $\<U\>$ is a non-elementary subgroup of $G$, 
we consider the $\mu_U$-random walk on $\Cay(G, U_0)$.

Then, we have
\[
\ell(\mu_U) \ge \kappa_1(\delta, m_0, m) L(U).
\]
\end{cor}

\proof
To show the claim, we assume that $1\in U$ and use Lemma \ref{lemma.lazy} if necessary.
We first show that there exists a positive integer $N$, which depends only 
on $\delta, m_0$ such that for every finite symmetric set $U$ with non-elementary $\<U\>$, 
we have
$L(U^N) \ge K_0 \delta +1$.
Indeed, since $1\in U$, the sequence 
$$U \subset U^2 \subset U^3 \subset \cdots$$
 is increasing. It is strictly increasing 
 since otherwise it would stabilize and $\langle U\rangle$ would be finite, 
which is impossible.
Hence $U^n$ contains at least $n$ points
for every integer $n>0$. 

Now define an integer $N:=\lceil{m_0^{K_0 \delta +1}\rceil}+1$. 
Let $x\in \Cay(G, U_0)$ be an arbitrary point.  
Since the set $U^N$ contains at least $N$ points, 
the set $U^Nx$  is not contained in the ball 
of radius $K_0 \delta +1 $ centered at $x$ 
in $\Cay(G,U_0)$ since the ball
contains at most $m_0^{K_0 \delta +1}$ points. 
This implies that  $L(U^N) \ge K_0 \delta +1$. 

As we noted before the statement, the action of $G$ on $\Cay(G,U_0)$ is $\Dc$-weakly acylindrical
for $\Dc=\Dc(\delta, m_0)$. 
We apply Theorem \ref{thm.drift2} to this action. 
Set $\kappa_1(\delta, m_0,m)=\kappa(\delta, \Dc, N,m)$, and 
we obtain the desired inequality. Here we used that both $\mathcal D, N$ depend
only on $\delta, m_0$. 
\qed

\medskip

Corollary \ref{cor.drift.hyp.gp} yields the following result on the drift spectrum of random walks on a fixed Cayley graph of a hyperbolic group.

\begin{cor}\label{cor.hyperbolic.spec}
Let $G$ be a non-elementary hyperbolic group, and $U_0$ a finite symmetric generating set of $G$. 
For each integer $m\ge 4$,
we define
$$
\mathcal{L}(G,U_0,m):=\Big\{\ell(\mu _U) \ : \ U\subset G, \ \# U \le m, \ U=U^{-1}, \ \text{$\<U\>$ is non-elementary} \Big\}.
$$
Then,
the set $\mathcal{L}(G,U_0,m)$
is discrete and unbounded.

In fact,
for each $A>0$ and each $m \ge 2$,
there are only finitely many symmetric $U$ up to conjugation by an element in $G$ with $\# U\le m$ such that $\<U\>$ is non-elementary and that $\ell(\mu_U) \le A$.

In particular, 
for each $A>0$, the set $\mathcal{L}(G,U_0,m) \cap (0,A]$ is finite. 
\end{cor}

\proof
Let $\delta$ denote the hyperbolicity constant of $\Cay(G, U_0)$, and $m_0:=\# U_0$.
We show that for each $A>0$ and for each $m \ge 2$,
there are only finitely many $U$ up to conjugation by an element in $G$ with $\# U=m$ such that $\<U\>$ is non-elementary and that $\ell(\mu_U) \le A$.
Since $\ell(\mu_U)$ is invariant under conjugating $U$ in $G$,
the set $\mathcal{L}(G, U_0, m) \cap (0, A]$ is finite.

Applying to Corollary \ref{cor.drift.hyp.gp},
we have a constant
$\kappa_1(\delta, m_0, m)>0$ 
such that
\[
\ell(\mu_U) \ge \kappa_1(\delta, m_0, m)L(U).
\]
Therefore, 
if $\ell (\mu_U) \le A$, then 
$L(U) \le A/\kappa_1(\delta, m_0, m)$.
Then, there exists an element $g$ in $G$ such that $g^{-1}Ug$ is contained in the ball of radius
$A/\kappa_1(\delta, m_0, m)+1$ centered at $1$ in $\Cay(G, U_0)$.
Thus, 
there are only finitely many such $U$ up to conjugation by an element in $G$.
In particular, this implies that for each  $A>0$,
the set $\mathcal{L}(G, U_0, m) \cap (0, A]$ is finite, so that the set $\mathcal{L}(G, U_0, m)$ is discrete.

Finally, the set is unbounded since $L(U)$ is unbounded if we vary $U$.
To see this, take two independent infinite order 
elements $g,h \in G$, set 
\[
U=\{g^n, g^{-n}, h^n, h^{-n}\}
\]
and let $n \to \infty$, so that the set $\mathcal{L}(G, U_0, m)$ is infinite and unbounded. 
\qed

\begin{remark}\label{rem.KKS}
Corollary \ref{cor.hyperbolic.spec} may be viewed as a partial generalization
of \cite[Theorem D]{KKS} on free groups to 
hyperbolic groups, possibly with torsion.
In \cite{KKS}, they discuss the set of values of {\it generic stretching factors} on finite rank non-abelian free groups (see their definition).

To explain the relation, fix a Cayley graph of a free group of rank $m$ associated to a free basis.
For the simple random walk associated with another free basis, the drift is equal to the generic stretching factor multiplied by $(m-1)/m$ (see \cite[Proof of Theorem 2.2, p.16]{KKS}).

Thus, in our terminology, \cite[Theorem D]{KKS} implies that the drifts obtained by varying the free basis belong to $\frac{1}{N}\mathbb Z$ for some positive integer $N$. We do not know  whether 
the set defined in 
Corollary \ref{cor.hyperbolic.spec} is contained in $r \Z$ for some $r \in \R$ for general hyperbolic groups.
\end{remark}

\subsection{Mapping class groups}\label{sec.MCG}

We refer the readers to \cite{MM} and \cite{Btight} 
for background on mapping class groups from the viewpoint used here.
Let $\mathcal S$ be a closed orientable surface of genus at least two. 
Let $\Mod(\mathcal S)$ be the mapping class group. It is the group of
isotopy classes of orientation preserving homeomorphisms of $\mathcal S$. This is a 
finitely presented group. 

Let $\mathcal C(\mathcal S)$ be the {\bf curve graph}.
It is a simplicial complex such that the  vertices are isotopy classes of essential simple curves on 
$\mathcal S$, and that two vertices  are joined by an edge if they are realized as disjoint simple curves. It is  connected, and we give the combinatorial path metric where every edge has length $1$. 
The curve graph $\mathcal C(\mathcal S)$ has infinite diameter and $\delta$-hyperbolic for some $\delta>0$ \cite[Theorem 1.1]{MM}.

We recall the Nielsen-Thurston classification: every  element $g \in \Mod(\mathcal S)$ is either of finite order, reducible, or {\bf pseudo-Anosov}. 
Here,  $g$ is {\bf reducible} if it has infinite order and 
preserves a finite collection of distinct isotopy classes of essential simple closed
curves on $\mathcal S$.

In the terminology of $\delta$-hyperbolic spaces,
an element $g \in \Mod(\mathcal S)$ is pseudo-Anosov if and only if $g$ acts on $\mathcal C(\mathcal S)$ as a hyperbolic isometry.
Otherwise, an element is elliptic in the sense that $g$ has a bounded orbit on 
$\mathcal C(\mathcal S)$. Moreover, if $g$ is elliptic, a finite power of $g$
fixes a vertex. See \cite[Proposition 4.6]{MM}.

The action of $\Mod(\mathcal S)$ on $\mathcal C(\mathcal S)$ is acylindrical \cite[Theorem 1.3]{Btight}. Hence, it is $\mathcal D$-weakly acylindrical, where the constant $\Dc$ depends only on $\mathcal S$,
\cite[Proposition 5]{D}.
We note that if $g$ is pseudo-Anosov, then $E(g)$ for the action of $\Mod(\mathcal S)$
on $\mathcal C(\mathcal S)$ is virtually cyclic. This is an immediate consequence
of the acylindricity of the action (cf.\ \cite[Lemma 6.5]{DGO}).

\begin{lemma}\label{lemma.mcg}
There exists a constant $N=N(\mathcal S)$ such that for 
every finite set $U$ in $\Mod(\mathcal S)$ with
$\<U\>$ containing a pseudo-Anosov element, 
we have
\[
L(U^N) \ge K_0\delta+1.
\]
Here $K_0$ is the absolute constant in Theorem \ref{thm.meta}, and $\delta$ is the hyperbolicity constant of $\mathcal C(\mathcal S)$.
\end{lemma}

\proof
It is known that there exists a constant $N_1=N_1(\mathcal S)$ such that 
if $\<U\>$ contains a pseudo-Anosov element, then
$U^{N_1}$ contains a pseudo-Anosov element $g$ \cite[Theorem 1.1]{M}.

Also, there is a constant $\tau=\tau(\mathcal S)>0$ such that $\tr u \ge \tau$ for every pseudo-Anosov
element $u$ in $\Mod(\mathcal S)$ \cite[Proposition 4.6]{MM}.
Hence, for every $n \in \N_0$,
we have $L(U^{N_1 n}) \ge \tr{g^n} \ge n\tau$ by Lemma \ref{lemma.disp} \eqref{item.trL}.
Letting $n:=\cl{(K_0\delta+1)/\tau}$ and $N:=N_1 n$,
we have
\[
L(U^N) \ge K_0\delta+1.
\]
We note that $N$ depends only on $\mathcal S$.
\qed

\medskip

We are ready to apply Theorem \ref{thm.drift2} to this setting.

\begin{thm}[Mapping class groups]\label{thm.mcg}
Let $\mathcal S$ be a closed orientable surface of genus at least two, and $m \in \N$. 
Let us consider the action of $\Mod(\mathcal S)$ on the curve graph $\mathcal C(\mathcal S)$ endowed with the combinatorial path metric.
Then, there exists a constant $\kappa(\mathcal S,m)>0$
such that the following holds:

Let $U$ be a finite symmetric set in $\Mod(\mathcal S)$ with $\#U=m$ such that $\<U\>$ contains a pseudo-Anosov element 
and not virtually cyclic. 
Consider the uniform distribution $\mu_U$ on $U$.
Then, we have
$$\ell(\mu_U) \ge \kappa(\mathcal S,m) L(U).$$
\end{thm}

\proof
The curve graph $\mathcal C(\mathcal S)$ is $\delta$-hyperbolic, and the action of $\Mod(\mathcal S)$ on $\mathcal C(\mathcal S)$ is $\mathcal D$-weakly acylindrical, where $\delta$ and $\Dc$ depend only on $\mathcal S$. 
By assumption, $\<U\>$ contains a pseudo-Anosov element.
By Lemma \ref{lemma.mcg}, 
there exists $N=N(\mathcal S) \in \N$, which depends only on $\mathcal S$, such that $L(U^{N}) \ge K_0 \delta +1$.

The elementary closure of every pseudo-Anosov element is virtually $\Z$.
Since by assumption $\<U\>$ contains a pseudo-Anosov element and is not virtually cyclic, 
there exists $u \in \<U\>$ such that $u \notin E(g)$.
Note that $g$ and $ugu^{-1}$ have disjoint sets of fixed points in the boundary of $\mathcal C(\mathcal S)$. Indeed, if otherwise, then the diameter of $\pi_{\qa{g}}(u\qa{g})$ is infinite, which contradicts the (weak) acylindricity of the $\Mod(\mathcal S)$-action on $\mathcal C(\mathcal S)$.
Thus, the action of $\<U\>$ on $\mathcal C (\mathcal S)$ is non-elementary.

Therefore, by Theorem \ref{thm.drift2}, we have
$$\ell(\mu_U) \ge \kappa(\delta, \mathcal D, N(\mathcal S), m) L(U).$$
Taking $\kappa(\mathcal S,m)=\kappa(\delta, \mathcal D, N(\mathcal S), m)$ gives the desired conclusion. 
\qed

\medskip

To state some consequences, we recall the classification of 
finitely generated subgroups of mapping class groups due to N.\ V.\ Ivanov \cite[7.15 Corollary]{I}. 
Similar to the classification of individual elements, such a subgroup
is one of the following: finite; reducible; virtually cyclic
and containing a pseudo-Anosov element; or containing two independent pseudo-Anosov elements. Here, a subgroup is {\bf reducible} if (it is infinite and) it preserves a finite 
collection of distinct isotopy classes of disjoint essential simple closed curves on $\mathcal S$.

Theorem \ref{thm.mcg} gives the following drift dichotomy
on finitely generated subgroups of mapping class groups.

\begin{cor}[Gap for $\ell(\mu_U)$ in $\Mod(\mathcal S)$]\label{cor.gap.mcg}
Let $\mathcal S$ be a closed orientable surface of genus at least two, and 
$m \in \N$. 
Then, there exists a constant $c(\mathcal S, m)>0$ such that the following holds:

For every finite set $U$ in $\Mod(\mathcal S)$ with $\# U=m$ and $U=U^{-1}$,
we have
\begin{enumerate}[(1)]
\item
$\ell(\mu_U) \ge c(\mathcal S, m)$, if
$\<U\>$ contains at least one pseudo-Anosov element, and $\< U\>$ is not virtually cyclic.

\item
Otherwise, $\ell(\mu_U)=0$.
\end{enumerate}

\end{cor}
\proof
(1) 
Since $\<U\>$ contains a pseudo-Anosov element, 
by Lemma \ref{lemma.mcg}, 
there exists $N=N(\mathcal S)$, which depends only on $\mathcal S$, such that $L(U^{N}) \ge K_0 \delta +1$, where $K_0$ is the absolute constant in Theorem \ref{thm.meta} and $\delta$ is the hyperbolicity constant of the curve graph $\mathcal C(\mathcal S)$.
Hence, by Theorem \ref{thm.mcg}, we have
\[
\ell(\mu_U) \ge \kappa(\mathcal S,m) L(U) \ge \kappa(\mathcal S, m)\frac{K_0 \delta+1}{N(\mathcal S)}.
\]
Letting $c(\mathcal S, m):=\kappa(\mathcal S, m)(K_0\delta+1)/N(\mathcal S)$ yields the desired inequality.

(2) We use the classification of subgroups of mapping class groups.
First, suppose that $\<U\>$ is finite, then its orbit is finite, hence clearly
$\ell(\mu_U)=0$. Second, suppose that 
$\< U\>$ is a reducible subgroup in $\Mod(\mathcal S)$. Then there is a finite set
of vertices $V$ in $\mathcal C(\mathcal S)$ that is invariant by $\< U\>$.
Again, the orbit by $\< U\>$ is finite, so that $\ell(\mu_U)=0$.
Finally, suppose $\<U\>$ contains a pseudo-Anosov element $g$, and $\<U\>$ is virtually cyclic.
Then, we have $\ell(\mu_U)=0$ (cf.\ the proof of Corollary \ref{cor.dichotomy}).
This completes the argument. 
\qed

\addcontentsline{toc}{section}{References}
\bibliographystyle{alpha}
\bibliography{drift}

\KF
\RT

\end{document}